\documentclass[12pt]{amsart}
\usepackage[english]{babel}
\usepackage[T1]{fontenc}
\usepackage[latin1]{inputenc}
\usepackage{lmodern}
\usepackage{a4}
\usepackage[all]{xy}
\usepackage{verbatim}
\usepackage{stmaryrd}
\usepackage{amsfonts}
\usepackage{amsmath,amssymb,amsthm}
\usepackage{enumerate}
\usepackage{xspace}
\usepackage{ulem}
\usepackage{euscript}
\usepackage{epsfig}
\usepackage{txfonts}
\usepackage{pxfonts}
\usepackage{mathrsfs}  
\usepackage{babel,indentfirst}
\usepackage[colorlinks,linkcolor=red,citecolor=blue]{hyperref}

\makeatletter\@addtoreset{equation}{section}

\newtheorem{theorem}{Theorem}[section]
\newtheorem{corollary}[theorem]{Corollary}
\newtheorem{lemma}[theorem]{Lemma}
\newtheorem{proposition}[theorem]{Proposition}
\theoremstyle{remark}
\newtheorem{definition}[theorem]{Definition}
\newtheorem{remark}[theorem]{Remark}
\newtheorem{example}[theorem]{Example}

\newcommand{\C}{\mathbb{C}}
\newcommand{\R}{\mathbb{R}}
\newcommand{\Zp}{\mathbb{Z}^+}
\newcommand{\N}{\mathbb{N}}
\renewcommand{\H}{\mathcal{H}}
\newcommand{\K}{\mathcal{K}}
\renewcommand{\l}{\mathcal{L}}
\newcommand{\M}{\mathcal{M}}
\renewcommand{\P}[1]{\mathcal{P}_{#1}(\mathbb{C})}
\newcommand{\lh}{{\mathcal B}(\mathcal{H})}
\newcommand{\lk}{{\mathcal B}(\mathcal{K})}
\newcommand{\csh}{{\rm CSO}}
\newcommand{\ssh}{{\rm SSO}}
\newcommand{\spn}{{\rm span}}
\newcommand{\ran}{{\rm ran}  }
\renewcommand{\ker}{{\rm ker}}

\title
{Interpolation theorems for conjugations and applications}

\author[Z. Amara]{Zouheir Amara}
\address{Department of Mathematics, Labo LIABM, Faculty of Sciences, Mohammed First University, 60000 Oujda, Morocco}
\email{z.amara@ump.ac.ma}



\begin{document}

\maketitle

\begin{abstract}
Let $\H$ be a separable complex Hilbert space. A conjugate-linear map $C:\H\to \H$ is called a conjugation if it is an involutive isometry. In this paper, we focus on the following interpolation problems: Let $\{x_i\}_{i\in I}$ and $\{y_i\}_{i\in I}$ be orthonormal sets of vectors in $\H$, and let $\{N_k\}_{k\in K}$ be a set of mutually commuting normal operators. We seek to determine under which conditions there exists a conjugation $C$ on $\H$ such that
\begin{enumerate}[\rm (a)]
\item  $Cx_i=y_i$ and $CN_kC=N_k^*$ for all $i\in I$ and $k\in K$; or
\item  $Cx_i=y_i$ and $CN_kC=-N_k^*$ for all $i\in I$ and $k\in K$.
\end{enumerate}
We provide complete answers to problems (a) and (b) using the spectral projections of normal operators. Our results are then applied to the study of complex symmetric and skew symmetric operators, as well as to the characterization of hyperinvariant subspaces of normal operators through conjugations. 
\end{abstract}

\section{Introduction}

Throughout this paper, $\H$ and $\K$ are separable complex Hilbert spaces. We use the notation $\langle ., .\rangle$ for the inner product of both $\H$ and $\K$. The algebra of all bounded linear operators on $\H$ is denoted by $\lh$, and the identity operator on $\H$ is denoted by $I_{\H}$. The range and the null space of an operator $T$ are denoted by $\ran(T)$ and $\ker(T)$, respectively. All direct sums in this paper are orthogonal. The sets of non-negative integers and positive integers are denoted, as customary, by $\N$ and $\Zp$, respectively. Throughout the paper, $\mathcal{N}=\{N_k\}_{k\in K}$ and $\mathcal{M}=\{M_k\}_{k\in K}$ are two sets of normal operators on $\H$ such that each two operators in their union are commuting, and $X=\{x_i\}_{i\in I}$ and $Y=\{y_i\}_{i\in I}$ will always denote two orthonormal sets of non-zero vectors in $\H$.
\begin{definition}\label{def.conj}
A {\it conjugation} $C$ on $\H$ is a conjugate-linear map $C:\H\to \H$ that satisfies
\begin{enumerate}[\rm (i)]
\item $C^2=I_{\H}$;
\item $\langle Ch,Ck\rangle=\langle k,h\rangle$ for all $h,k\in \H$.
\end{enumerate}
\end{definition}
The simplest example of a conjugation on $\H$ is the map given by
$
C\left(\sum_i \lambda_i e_i\right)=\overline{\lambda_i}e_i
$
where $\{e_i\}_i$ is an orthonormal basis of $\H$. In fact, according to \cite[Lemma 1]{Garcia.Putinar}, any conjugation can be represented in this manner with respect to some orthonormal basis.

Conjugations have been the subject of intensive study in recent years  \cite{Camara.Klis.Lanucha.Ptak, Camara.Klis.Lanucha.Ptak2, Dymek.Planeta.Ptak, Dymek.Planeta.Ptak2, Mashreghi.Ptak.RossI, Mashreghi.Ptak.RossII}, with origins in physics. The composition $\mathcal{P}\mathcal{T}$, of the {\it parity} operator $[\mathcal{P}f](x)=f(-x)$ and the {\it time-reversal} operator $[\mathcal{T}f](x)=\overline{f(x)}$ defines a conjugation on $L^2(\R^n)$ which holds significant importance in the theory of $\mathcal{P}\mathcal{T}$-symmetric quantum theory (see, for instance, \cite{Bender,Bender.Boettcher}). For more historical comments about conjugations, we refer the interested reader to \cite{Camara.Klis.Lanucha.Ptak, Wang.Xie.Yan.Zhu} and references therein. Recent investigations into conjugations have largely stemmed from their connections to complex-symmetric and skew-symmetric operators.

An operator $T\in\lh$ is called {\it complex-symmetric} (resp. {\it skew-symmetric}) if there exists a conjugation $C$ on $\H$ such that $CTC = T^*$ (resp. $CTC = -T^*$); in this case, we say more precisely that $T$ is {\it $C$-symmetric} (resp. {\it $C$-skew symmetric}). An equivalent definition is that $T$ has a symmetric (resp. skew-symmetric) matrix $M$ with respect to some orthonormal basis; i.e., $M^{\rm tr}=M$ (resp. $M^{\rm tr}=-M$) where ${\rm tr}$ stands for the transpose (\cite{Garcia.Putinar2, Li.Zhu}). Denote by $\csh$ and $\ssh$ the classes of complex symmetric and skew symmetric operators on $\H$, respectively. These classes of linear operators have received significant attention in the last two decades. Applications of complex symmetric operators and skew-symmetric matrices are found in various areas (see, for instance, \cite{Garcia.Putinar, Garcia.Putinar2, Garcia.Prodan.Putinar, Prodan.Garcia.Putinar} and \cite{Cao.Hu, Mehl, Pinero.Singh}).

\medskip

Given two vectors $x,y\in \H$, we denote by $x\otimes y$ the operator defined by $(x\otimes y)h=\langle h,y\rangle x$ for every $h\in \H$. It is easy to see that if $C$ is a conjugation, then
\begin{equation}\label{CxoyC}
C(x\otimes y)C=(Cx)\otimes(Cy)\quad\mbox{for all }x,y\in \H.
\end{equation}

In \cite{Liu.Shi.Wang.Zhu, Liu.Xie.Zhu, Wang.Xie.Yan.Zhu, Zhu.Li}, the authors investigated various interpolation problems for conjugations. For instance, in \cite{Liu.Shi.Wang.Zhu}, Liu et al. established that for two orthogonal projections $P$ and $Q$ in $\lh$, a conjugation $C$ on $\H$ exists with $CPC=Q$ if and only if $\ran(P)$ and $\ran(Q)$ are in symmetric position; i.e.,
$$
\dim\big(\ran(P)^{\perp}\cap \ran(Q)\big)=\dim\big(\ran(P)\cap \ran(Q)^{\perp}\big).
$$


The main motivation for our results in Section \ref{interpolation.section.Csym} stems from the study of complex symmetric operators. It should be mentioned that the class $\csh$ is not stable under the addition. Indeed, it is well-known that every normal operator $N\in\lh$ is complex symmetric. However, even its sum with the simplest complex symmetric operator may fail to be in $\csh$ (see \cite[Theorem 2.4]{Amara.Oudghiri}). The main results in Section \ref{interpolation.section.Csym} can be used to study complex symmetry of $N+T$ for a dense class of operators $T$ in $\lh$. To elaborate, Zhu and Li proved in \cite{Zhu.Li} that operators $T\in\lh$ of the form
\begin{equation}\label{Dense.class}
T=\sum_{j\in J} r_j e_j\otimes f_j,
\end{equation}
where $\{r_j\}_{j\in J}$ are distinct positive scalars and $\{e_j\}_{j\in J}$ and $\{f_j\}_{j\in J}$ are orthonormal subsets of $\H$, form a dense class in $\lh$. Moreover, they showed that such an operator is complex symmetric if and only if there exist unimodular scalars $\lambda_j\in\C$, $j\in J$, such that
\begin{equation}\label{geo.cri}
\lambda_i\langle e_i,f_j\rangle=\lambda_j\langle e_j,f_i\rangle\quad\mbox{for all }i,j\in J.
\end{equation}
Corollary \ref{cmplx.sym.N+T} provides geometric conditions for the existence of a conjugation $C$ such that $N_k+\lambda T$ is $C$-symmetric for all $k\in K$ and $\lambda\in\C$. When $N_k=0$ for every $k\in K$, we recover the result in \eqref{geo.cri} with ease. Furthermore, our results extend a recent finding by Mashreghi et al. \cite{Mashreghi.Ptak.RossI} on hyperinvariant subspaces of unitary operators (see Proposition \ref{hyperinvariant}).

In Section \ref{interpolation.section.Cskew}, we establish necessary and sufficient conditions for the existence of a conjugation $C$ such that $CN_kC=-N_k^*$ and $Cx_i=y_i$ for all $k\in K$ and $i\in I$. These results are used to characterize skew symmetry of operators of the form
\begin{equation}\label{Dense.class.skew}
N+\sum_{j\in J} r_j \left(e_j\otimes e_j - f_j\otimes f_j\right),
\end{equation}
where $N\in\lh$ is normal, $\{r_j\}_{j\in J}$ are distinct positive scalars and $\{e_j,f_j\}_{j\in J}$ is an orthonormal subset. Notably, skew symmetric operators of the form \eqref{Dense.class.skew} are dense in $\ssh$ (see Section \ref{interpolation.section.Cskew}).

To prove the main results of Section \ref{interpolation.section.Cskew}, we will uncover certain facts about skew symmetric normal operators. We will show that $\{N_k\}_{k\in K}$ are simultaneously skew-symmetric (i.e., there exists a conjugation $C$ such that $CN_kC=-N_k^*$ for every $k\in K$) if and only if there exists a unitary operator $U\in\lh$ such that $UN_kU^*=-N_k$ for every $k\in K$. Furthermore, for a normal operator $N\in\lh$, we provide a condition on the multiplicity of $N$ under which skew symmetry of $N\oplus T$ implies that one of $T$, where $T$ is an arbitrary operator.


\section{Interpolation theorems for conjugations : $C$-symmetric versions}\label{interpolation.section.Csym}
Given a subset $\mathcal{A}=\{ A_k \}_{k\in K}\subset \lh$ of mutually commuting normal operators and a subset of vectors $Z\subseteq\H$, we denote
$$
L_{\mathcal{A},Z} =\bigvee\big\{ \prod_{k\in K_0} A_k^{n_k} A_k^{m_k*}h : K_0\subseteq K\mbox{ is finite, }n_k,m_k\in\N \mbox{ and } h\in Z \big\}
$$
where $\vee$ denotes the closed linear span. By the Fuglede Theorem, we also have $A_kA_l^*=A_l^*A_k$ for any $k,l\in K$. Hence, one can see that $L_{\mathcal{A},Z}$ is a reducing subspace of each $A_k$.

\medskip

For a normal operator $N$, we denote by $E_N$ the associated spectral measure defined on Borel subsets of $\C$. The following theorem is central to proving our main results.

\subsection{Main results}

%

\begin{theorem}\label{interpo.Nk.Mk}
Assume that $\H=L_{\mathcal{N}\cup\mathcal{M},X\cup Y}$. The following statements are equivalent:
\begin{enumerate}[\rm (i)]

\item There exists a conjugation $C$ on $\H$ such that $CN_kC=M_k^*$ and $Cx_i=y_i$ for all $k\in K$ and $i\in I$.

\item For all finite subsets $K_0,K_1\subseteq K$, Borel subsets $\{\Delta_k\}_{k\in K_0}$ and $\{D_k\}_{k\in K_1}$ of $\C$, and $i,j\in I$, we have
$$
\left\langle \left(\prod_{k\in K_0} E_{N_k}(\Delta_k)\right) x_i, \left(\prod_{l\in K_1} E_{M_l}(D_l)\right)x_j \right\rangle
=
\left\langle \left(\prod_{l\in K_1} E_{N_l}(D_l)\right) y_j, \left(\prod_{k\in K_0} E_{M_k}(\Delta_k)\right)y_i \right\rangle
$$
and
$$
\left\langle \left(\prod_{k\in K_0} E_{N_k}(\Delta_k)\right) x_i, \left(\prod_{l\in K_1} E_{M_l}(D_l)\right)y_j \right\rangle
=
\left\langle \left(\prod_{l\in K_1} E_{N_l}(D_l)\right) x_j, \left(\prod_{k\in K_0} E_{M_k}(\Delta_k)\right)y_i \right\rangle.
$$
\end{enumerate}
\end{theorem}

\medskip
\begin{remark}
Recall that an {\it anti-automorphism} $\rho:\lh\to\lh$ is a vector space automorphism satisfying $\rho(T^*)=\rho(T)^*$ and $\rho(TS)=\rho(S)\rho(T)$ for all $T,S\in\lh$. It is well-known that each anti-automorphism $\rho:\lh\to\lh$, of order two (i.e., $\rho^2=I_{\lh}$), has the form $\rho(T)=CT^*C$ for every $T\in\lh$, where $C$ is a conjugation on $\H$. In the light of this fact, it is not hard to see that Theorem \ref{interpo.Nk.Mk} can be used to obtain conditions under which there exists an anti-automorphism $\rho:\lh\to\lh$ of order two such that
$$
\rho(T)=T
\quad\mbox{and}\quad
\rho(N_k)=M_k\;\mbox{for every }k\in K,
$$
where $T$ is an operator of the form \eqref{Dense.class}.
\end{remark}

\medskip

For a subspace $M$ of $\H$, denote by $P_M\in\lh$ the orthogonal projection of $\H$ onto $M$. The following corollary is a direct consequence of Theorem \ref{interpo.Nk.Mk}.

\begin{corollary}\label{interpo.subspaces}
Assume that $\H=\vee\{x_i,y_i : i\in I\}$, and let $M$ and $N$ be subspaces of $\H$ such that $P_M$ and $P_N$ commute. The following are equivalent:
\begin{enumerate}[\rm (i)]
\item There exists a conjugation $C$ on $\H$ such that $CM=N$ and $Cx_i=y_i$ for every $i\in I$.
\item For all $i,j\in I$ and $r,s\in\{0,1\}$,  we have
$$
\left\langle P_N^rP_M^s x_i, x_j \right\rangle=\left\langle P_N^sP_M^r y_j, y_i \right\rangle
\quad\mbox{and}\quad
\left\langle P_N^rP_M^s x_i, y_j \right\rangle=\left\langle P_N^sP_M^r x_j, y_i \right\rangle.
$$
\end{enumerate}
\end{corollary}

\medskip

Recall that a linear (or conjugate-linear) operator $V$ on $\H$ is called {\it partial isometry} if the restriction of $V$ to $\ker(V)^{\perp}$ is an isometry.

\medskip

If $N_k=M_k$ for every $k\in K$, Theorem \ref{interpo.Nk.Mk} becomes:

\begin{theorem}\label{interpo.Nk.Csym}
The following statements are equivalent:
\begin{enumerate}[\rm (i)]

\item There exists a conjugation $C$ on $\H$ such that $CN_kC=N_k^*$ and $Cx_i=y_i$ for all $k\in K$ and $i\in I$.

\item There exists a conjugate-linear partial isometry $V$ on $\H$ such that $VN_k=N_k^*V$, $x_i\in\ker(V)^{\perp}$, $Vx_i=y_i$ and $Vy_i=x_i$ for all $k\in K$ and $i\in I$.

\item For all finite subsets $K_0\subseteq K$, Borel subsets $\{\Delta_k\}_{k\in K_0}$ of $\C$,  and $i,j\in I$ we have
$$
\left\langle \left(\prod_{k\in K_0} E_{N_k}(\Delta_k)\right) x_i,x_j \right\rangle
=
\left\langle \left(\prod_{k\in K_0} E_{N_k}(\Delta_k)\right) y_j,y_i \right\rangle
$$
and
$$
\left\langle \left(\prod_{k\in K_0} E_{N_k}(\Delta_k)\right)x_i,y_j \right\rangle
=
\left\langle \left(\prod_{k\in K_0} E_{N_k}(\Delta_k)\right) x_j,y_i \right\rangle.
$$
\end{enumerate}
\end{theorem}

\medskip

\begin{remark}
Taking $N_k=0$, for every $k\in K$, in Theorem \ref{interpo.Nk.Csym}, we recapture \cite[Theorem 2.1]{Zhu.Li} which states that there exists a conjugation $C$ on $\H$ satisfying $Cx_i=y_i$ for every $i\in I$ if and only if
$
\langle x_i,y_j\rangle=\langle x_j,y_i\rangle
$
for all $i,j\in I$.
\end{remark}

\begin{remark}
One can use Theorem \ref{interpo.Nk.Csym} to show that if the subspaces $M$ and $N$ in Corollary \ref{interpo.subspaces} are equal, then the  condition ``$\H=\vee\{x_i,y_i : i\in I\}$'' is no longer needed.
\end{remark}

\medskip

\begin{corollary}\label{cmplx.sym.N+T}
Let $\{r_i\}_{i\in I}$ be a set of distinct positive scalars, and let $T=\sum_{i\in I}r_i x_i\otimes y_i$. The following statements are equivalent:
\begin{enumerate}[\rm (i)]
\item There exists a conjugation $C$ on $\H$ such that $N_k+\lambda T$ is $C$-symmetric for all $k\in K$ and $\lambda\in\C$.
\item There exist unimodular scalars $\lambda_i\in\C$, for $i\in I$, such that
$$
\lambda_i\left\langle \left(\prod_{k\in K_0} E_{N_k}(\Delta_k)\right) x_i,x_j \right\rangle
=
\lambda_j\left\langle \left(\prod_{k\in K_0} E_{N_k}(\Delta_k)\right) y_j,y_i \right\rangle
$$
and
$$
\lambda_i\left\langle \left(\prod_{k\in K_0} E_{N_k}(\Delta_k)\right)x_i,y_j \right\rangle
=
\lambda_j\left\langle \left(\prod_{k\in K_0} E_{N_k}(\Delta_k)\right) x_j,y_i \right\rangle
$$
for all finite subsets $K_0\subseteq K$, Borel subsets $\{\Delta_k\}_{k\in K_0}$ of $\C$, and $i,j\in I$.
\end{enumerate}
\end{corollary}

\begin{proof}
(ii)$\implies$(i). For every $i\in I$, put $\widetilde{x_i}=\lambda_i x_i$. It is easy to see that the orthonormal sets $\{\widetilde{x_i}\}_{i\in I}$ and $\{y_i\}_{i\in I}$ satisfy the assertion (iii) in Theorem \ref{interpo.Nk.Csym}. Thus, there exists a conjugation $C$ on $\H$ such that $CN_kC=N_k^*$ and $C\widetilde{x_i}=y_i$ for all $k\in K$ and $i\in I$. Moreover, since $C^{-1}=C$, we also have $Cy_i=\widetilde{x_i}$. Now, by \eqref{CxoyC}, one can check that $C(x_i\otimes y_i)C=y_i\otimes x_i$ for every $i\in I$. Therefore, $CTC=T^*$, and consequently, $C(N_k+\lambda T)C=(N_k+ \lambda T)^*$ for all $\lambda\in\C$ and $k\in K$.

\medskip

(i)$\implies$(ii). Since the class of $C$-symmetric operators is a vector space and $\lambda$ is arbitrary, we have $CN_kC=N_k^*$ for every $k\in K$ and $CTC=T^*$. The later equality yields $CT^*TC=(CT^*C)(CTC)=TT^*$, and consequently $C(T^*T-r_i I_{\H})C=TT^*-r_i I_{\H}$ for every $i\in I$. In particular, $C\ker(TT^*-r_i I_{\H})\subseteq\ker(T^*T-r_i I_{\H})$ for every $i\in I$. Therefore, $C\left(\vee\{x_i\}\right)=\vee\{y_i\}$ for every $i\in I$. Since $C$ is isometric, there exist unimodular scalars $\lambda_i$, for $i\in I$, such that $C(\lambda_i x_i)=y_i$ for every $i\in I$. Now, using Theorem \ref{interpo.Nk.Csym}, one can easily see that the scalars $\lambda_i$, $i\in I$, satisfy all the desired equalities.
\end{proof}

If $I$ is of cardinal one, then we get a linear version of Theorem \ref{interpo.Nk.Csym} as follows:

\begin{theorem}\label{interpo.Csym}
Let $x,y\in \H$ be unit vectors. The following statements are equivalent:
\begin{enumerate}[\rm (i)]
\item There exists a conjugation $C$ on $\H$ satisfying $CN_kC=N_k^*$ and $Cx=y$ for every $k\in K$.
\item There exists a unitary operator $U\in\lh$ satisfying $UN_k=N_kU$ and $Ux=y$ for every $k\in K$.
\item There exists a partial isometry $V\in\lh$ satisfying $x\in\ker(V)^{\perp}$, $VN_k=N_kV$ and $Vx=y$ for every $k\in K$.
\item For all finite subsets $K_0\subseteq K$ and Borel subsets $
\{\Delta_k\}_{k\in K_0}$ of $\C$, we have
$$
\left\Vert \left(\prod_{k\in K_0} E_{N_k}(\Delta_k)\right)x\right\Vert
=
\left\Vert \left(\prod_{k\in K_0} E_{N_k}(\Delta_k)\right)y\right\Vert.
$$
\item There exists a conjugation $C$ on $\H$ such that $N_k+\lambda x\otimes y$ is $C$-symmetric for all $k\in K$ and $\lambda\in\C$.
\end{enumerate}
\end{theorem}

\medskip

The proofs of Theorems \ref{interpo.Nk.Mk}, \ref{interpo.Nk.Csym}, and \ref{interpo.Csym} will be given later in this section, following the necessary preliminary preparations.

\medskip

\begin{remark}
A slight reformulation of \cite[Theorem 2.1]{Amara.Oudghiri} shows that if $N \in \lh$ is a normal operator, $U \in \lh$ is a unitary operator commuting with $N$, and $x \in \H$, then for every $\lambda \in \C$, there exists a conjugation $C$ on $\H$ such that $N+\lambda x \otimes Ux$ is $C$-symmetric. It can be seen that the implication (ii)$\implies$(v) in Theorem \ref{interpo.Csym} generalizes this result, as the conjugation $C$ is independent of $\lambda$.
\end{remark}

\medskip

A subspace $M$ of $\H$ is said to be {\it hyperinvariant} for $T\in\lh$ if $M$ is invariant under every operator that commutes  with $T$. Recently, in \cite{Mashreghi.Ptak.RossI}, Mashreghi et all. proved that $M$ is hyperinvariant for a unitary operator $U\in\lh$ if and only if $CM\subseteq M$ for every conjugation $C$ on $\H$ satisfying $CUC=U^*$. As a consequence of Theorem \ref{interpo.Csym}, this result is extended to all normal operators.

\begin{proposition}\label{hyperinvariant}
Let $N\in\lh$ be normal and $M$ a subspace of $\H$. The following statements are equivalents:
\begin{enumerate}[\rm (i)]
\item $M$ is hyperinvariant for $N$.
\item $TM\subseteq M$ for every conjugate-linear operator $T$ on $\H$ satisfying $TN=N^*T$.
\item $CM\subseteq M$ for every conjugation $C$ on $\H$ satisfying $CNC=N^*$.
\end{enumerate}
\end{proposition}

\begin{proof}
(i)$\implies$(ii). Assume that $M$ is hyperinvariant for $N$, and let $T$ be a conjugate-linear operator on $\H$ such that $TN=N^*T$. Clearly, it suffices to show that $TP_{M}=P_{M}T$. According to \cite[Proposition 6.9]{Radjavi.Rosenthal}, there exists a Borel subset $\Delta\subseteq\C$ such that $P_{M}=E_N(\Delta)$. Consequently, Lemma \ref{TN=M*T} implies that $TP_{M}=TE_N(\Delta)=E_N(\Delta)T=P_{M}T$.

\medskip

(ii)$\implies$(iii) If $C$ is a conjugation on $\H$ such that $CNC=N^*$, and since $C^{-1}=C$, it follows that $CN=N^*C$, which leads to the desired containment.

\medskip

(iii)$\implies$(i). We need to show that for each $T\in\{N\}'$, the commutant of $N$, $TM\subseteq M$. Additionally, since $\{N\}'$ is a von Neumann algebra, it is spanned by its unitary operators (\cite[p. 61, Proposition 13.3 (b)]{Conway.Op.Th}). Consequently, we can assume without loss of generality that $T$ is unitary. Let $x\in M$ be non-zero, and put $y=Tx$. Then, according to Theorem \ref{interpo.Csym}, there exists a conjugation $C$ on $\H$ satisfying $CNC=N^*$ and $Cx=y$. Hence, by hypothesis, $Tx=y=Cx\in M$. Since $x$ was arbitrary, we conclude that $TM\subseteq M$.
\end{proof}

The following example shows that the result in Proposition \ref{hyperinvariant} cannot be extended to an arbitrary operator in $\lh$.

\begin{example}
Let $x$ and $y$ be orthonormal vectors in $\H$, and put $T=A+{\rm i}B$ where $A=x\otimes x$ and $B=(x+y)\otimes (x+y)$. Clearly, $A$ and $B$ are self-adjoint operators; moreover, $A=\frac{1}{2}(T+T^*)$. It follows that if $C$ is a conjugation on $\H$ such that $CTC=T^*$, then $CAC=A$. Hence, if $M=\vee\{x\}$, then $CM=M$. Therefore, the subspace $M$ satisfies  condition (iii) in Proposition \ref{hyperinvariant}; however, $T\in\{T\}'$ and $TM\nsubset M$.
\end{example}

\medskip

Given a conjugation $C$ on $\H$, an operator $T\in\lh$ is said to be {\it $C$-normal} if $CTT^*C=T^*T$. When $T$ is $C$-normal for some conjugation $C$, we say that $T$ is {\it conjugate-normal}. It is easy to verify that if $T$ is $C$-symmetric or $C$-skew symmetric, then it is also $C$-normal. Consequently, conjugate-normality is considered a more general type of symmetry. For further details and properties of such operators, the interested reader is referred to \cite{Wang.Zhao.Zhu}.

\medskip

Another consequence of Theorem \ref{interpo.Csym} is the following result, which characterizes normal operators in terms of complex symmetric and conjugate-normal ones.

\begin{corollary}
Let $T\in\lh$. The following statements are equivalent:
\begin{enumerate}[\rm (i)]
\item $T$ is normal.
\item For every unit vector $x\in \H$, there exists a conjugation $C$ on $\H$ such that $T$ is $C$-symmetric and $Cx=x$.
\item For every unit vector $x\in \H$, there exists a conjugation $C$ on $\H$ such that $T$ is $C$-normal and $Cx=x$.
\end{enumerate}
\end{corollary}

\begin{proof}
The implication (i)$\implies$(ii) follows directly from Theorem \ref{interpo.Csym} applied to the normal operator $T$ and the vectors $x=y$.

(ii)$\implies$(iii) is obvious.

(iii)$\implies$(i). Let $x\in \H$ be a unit vector. Then, by hypothesis, there exists a conjugation $C$ on $\H$ such that $Cx=x$ and $CTT^*C=T^*T$. It follows that
\begin{align*}
\langle T^*Tx,x\rangle = \langle CTT^*Cx,x\rangle=\langle Cx,CCTT^*Cx \rangle=\langle x,TT^*x \rangle=\langle TT^*x,x \rangle.
\end{align*}
Since $x$ was arbitrary, we infer that $T^*T=TT^*$, the desired equality.
\end{proof}

\subsection{Proof of Theorems \ref{interpo.Nk.Mk}, \ref{interpo.Nk.Csym}, and \ref{interpo.Csym}}

The proofs require some preparation and will be provided at the end of this section. In what follows, we denote the set of all complex polynomials in $n$  variables by $\P{n}$. We begin with the following technical lemma.

\begin{lemma}\label{E.poly}
Let $\{ N_i \}_{1\leq i\leq n} \subset\lh$ and $\{ M_i \}_{1\leq i\leq n}\subset\mathcal{B}(\K)$ be two finite sets of mutually commuting normal operators, and let $x,y\in \H$ and $h,k\in \K$ be vectors such that
$$
\left\langle \prod_{i=1}^n \left(E_{N_i}(\Delta_i)\right) x,y \right\rangle
=
\left\langle\prod_{i=1}^n \left(E_{M_i}(\Delta_i)\right) h,k\right\rangle\quad\mbox{for every Borel subset $\Delta_i\subseteq\C$}.
$$
Then, for every $p\in\P{2n}$, we have
$$
\left\langle p(N_1,N_1^*, \dots ,N_n,N_n^*)x,y \right\rangle=\left\langle p(M_1,M_1^*, \dots ,M_n,M_n^*) h,k\right\rangle.
$$
\end{lemma}

\begin{proof}
Let $p\in\P{2n}$. By the commutativity assumption, we may suppose that $p(\alpha_1,\lambda_1,\dots, \alpha_n,\lambda_n)=\prod_{j=1}^{n}\alpha_i^{r_i}\lambda_i^{s_i}$ where $r_i,s_i\in\N$. For each $i\in\{ 1,\dots,n\}$, put $f_i(z)=z^{r_i}\overline{z}^{s_i}\chi_{K_i}(z)$ where $K_i=\sigma(N_i)\cup\sigma(M_i)$ and $\chi_{K_i}$ is the characteristic function of $K_i$. Then, for every $i\in\{ 1,\dots,n\}$, $f_i$ is a bounded Borel function on $\C$. Additionally,
\begin{equation}\label{fi(Ni),fi(Mi)}
f_i(N_i)=N_i^{r_i}N_i^{*s_i}
\quad\mbox{and}\quad
f_i(M_i)=M_i^{r_i}M_i^{*s_i}
\quad\mbox{for every }i\in\{ 1,\dots,n\}.
\end{equation}

Fix Borel subsets $\{\Delta_i\}_{2\leq i\leq n}\subseteq\C$ and denote by $\mu_1$ and $\nu_1$ the complex-valued measures on Borel subsets of $\C$ given by
$$
\mu_1(\Delta)=\left\langle E_{N_1}(\Delta) \prod_{i=2}^n \left(E_{N_i}(\Delta_i)\right)x,y\right\rangle
\quad\mbox{and}\quad
\nu_1(\Delta)=\left\langle E_{M_1}(\Delta) \prod_{i=2}^n \left(E_{M_i}(\Delta_i)\right)h,k\right\rangle.
$$
Then,
$$
\left\langle N_1^{r_1}N_1^{*s_1} \prod_{i=2}^n \left(E_{N_i}(\Delta_i)\right)x,y \right\rangle
=\left\langle f_1(N_1) \prod_{i=2}^n \left(E_{N_i}(\Delta_i)\right)x,y\right\rangle=\int f_1{\,\rm d}\mu_1
$$
and
$$
\left\langle M_1^{r_1}M_1^{*s_1} \prod_{i=2}^n \left(E_{M_i}(\Delta_i)\right)h,k \right\rangle
=\left\langle f_1(M_1) \prod_{i=2}^n \left(E_{M_i}(\Delta_i)\right)h,k\right\rangle=\int f_1{\,\rm d}\nu_1.
$$
As the measures $\mu_1$ and $\nu_1$ are presumed equal by hypothesis, we deduce
$$
\left\langle N_1^{r_1}N_1^{*s_1} \prod_{i=2}^n \left(E_{N_i}(\Delta_i)\right)x,y \right\rangle
=\left\langle M_1^{r_1}M_1^{*s_1} \prod_{i=2}^n \left(E_{M_i}(\Delta_i)\right)h,k \right\rangle.
$$
By the Fuglede Theorem, we obtain $\mu_2(\Delta_2)=\nu_2(\Delta_2)$ where
$$
\mu_2(\Delta_2):=\left\langle E_{N_2}(\Delta_2) N_1^{r_1}N_1^{*s_1} \prod_{i=3}^n \left(E_{N_i}(\Delta_i)\right)x,y \right\rangle
$$
and
$$
\nu_2(\Delta_2):=\left\langle E_{M_2}(\Delta_2) M_1^{r_1}M_1^{*s_1} \prod_{i=3}^n \left(E_{M_i}(\Delta_i)\right)h,k \right\rangle.
$$
Since $\Delta_2$ was arbitrary, the complex-valued measures $\mu_2$ and $\nu_2$ are equal. Integrating $f_2$ with respect to these measures and taking into account \eqref{fi(Ni),fi(Mi)}, we obtain
$$
\left\langle N_2^{r_2}N_2^{*s_2} N_1^{r_1}N_1^{*s_1} \prod_{i=3}^n \left(E_{N_i}(\Delta_i)\right)x,y \right\rangle
=\left\langle M_2^{r_2}M_2^{*s_2}  M_1^{r_1}M_1^{*s_1} \prod_{i=3}^n \left(E_{M_i}(\Delta_i)\right)h,k \right\rangle.
$$
The desired equality will follow by continuing the previous process.
\end{proof}

\medskip

\begin{proposition}\label{Wisom}
Let $\mathcal{A}=\{A_k\}_{k\in K}\subset\lh$ and $\mathcal{B}=\{B_k\}_{k\in K}\subset\lk$ be two sets of mutually commuting normal operators, and let $A=\{a_i\}_{i\in I}\subseteq \H$ and  $B=\{b_i\}_{i\in I}\subseteq \K$ be two subsets of vectors such that
$$
\left\langle \left(\prod_{k\in K_0} E_{A_k}(\Delta_k)\right) a_i,a_j \right\rangle
=\left\langle \left(\prod_{k\in K_0} E_{B_k}(\Delta_k)\right)b_j,b_i \right\rangle
$$
for all finite subsets $K_0\subseteq K$, Borel subsets $\{\Delta_k \}_{k\in K_0}$ of $\C$ and $i,j\in I$.

Then, there exists a conjugate-linear surjective isometry $W:L_{\mathcal{A},A}\to L_{\mathcal{B},B}$ such that
$$
Wa_i=b_i\quad\mbox{and}\quad Wp(T_{k_1},T_{k_1}^*,\dots, T_{k_n},T_{k_n}^* )=p(S_{k_1},S_{k_1}^*,\dots, S_{k_n},S_{k_n}^* )^*W
$$
for all $i\in I$, $n\in\Zp$, $p\in\P{2n}$ and $\{k_l\}_{1\leq l\leq n}\subseteq K$, where $T_{k_i}$ and $S_{k_i}$ are the restrictions of $A_{k_i}$ and $B_{k_i}$ to $L_{\mathcal{A},A}$ and $L_{\mathcal{B},B}$, respectively.
\end{proposition}

\begin{proof}
According to Lemma \ref{E.poly}, for every subset $\{k_l\}_{1\leq l\leq n}\subseteq K$, every polynomial $p\in\P{2n}$, and all $i,j\in I$, we have
\begin{equation}\label{p.xi.xj}
\left\langle p(A_{k_1},A_{k_1}^*,\dots,A_{k_n},A_{k_n}^*)a_i,a_j\right\rangle
=
\left\langle p(B_{k_1},B_{k_1}^*,\dots,B_{k_n},B_{k_n}^*)b_j,b_i\right\rangle.
\end{equation}
For convenience, put
$$
\H_0=\spn\big\lbrace p(A_{k_1},A_{k_1}^*,\dots,A_{k_n},A_{k_n}^*)a_i: n\in\Zp, p\in\P{2n}, \{k_l\}_{1\leq l\leq n}\subseteq K, i\in I\big\rbrace
$$
and
$$
\K_0=\spn\big\lbrace p(B_{k_1},B_{k_1}^*,\dots,B_{k_n},B_{k_n}^*)b_i: n\in\Zp, p\in\P{2n}, \{k_l\}_{1\leq l\leq n}\subseteq K , i\in I\big\rbrace,
$$
and let $W:\H_0\to \K_0$ be the conjugate-linear transformation given by 
$$
W\left(p(A_{k_1},A_{k_1}^*,\dots,A_{k_n},A_{k_n}^*)a_i \right)= p(B_{k_1},B_{k_1}^*,\dots,B_{k_n},B_{k_n}^*)^*b_i
$$
for all $n\geq 1$, $p\in\P{2n}$, $\{k_l\}_{1\leq l\leq n}\subseteq K$ and $i\in I$.

We will first show that $W$ is well-defined. Arbitrarily choose $h\in \H_0$; then, there exists a finite subset $J$ of $I$, $\{k_{i,l}\}_{1\leq l\leq n_i}\subseteq K$, for $i\in J$, and polynomials $p_i\in\P{2n_i}$, for $i\in J$, such that
$$
h=\sum_{i\in J}p_i(A_{k_{i,1}},A_{k_{i,1}}^*,\dots,A_{k_{i,n_i}},A_{k_{i,n_i}}^*)a_i.
$$
Then,
\begin{align*}
\Vert h\Vert^2
&=\left\langle \sum_{i\in J}p_i(A_{k_{i,1}},A_{k_{i,1}}^*,\dots,A_{k_{i,n_i}},A_{k_{i,n_i}}^*)a_i , \sum_{j\in J}p_j(A_{k_{j,1}},A_{k_{j,1}}^*,\dots,A_{k_{j,n_j}},A_{k_{j,n_j}}^*)a_j
\right\rangle
\\
&=\sum_{i\in J} \sum_{j\in J} \left\langle p_i(A_{k_{i,1}},A_{k_{i,1}}^*,\dots,A_{k_{i,n_i}},A_{k_{i,n_i}}^*)a_i , p_j(A_{k_{j,1}},A_{k_{j,1}}^*,\dots,A_{k_{j,n_j}},A_{k_{j,n_j}}^*)a_j
\right\rangle
\\
&=\sum_{i\in J} \sum_{j\in J} \left\langle p_i(A_{k_{i,1}},A_{k_{i,1}}^*,\dots,A_{k_{i,n_i}},A_{k_{i,n_i}}^*) p_j(A_{k_{j,1}},A_{k_{j,1}}^*,\dots,A_{k_{j,n_j}},A_{k_{j,n_j}}^*)^* a_i , a_j
\right\rangle.
\end{align*}
It follows by \eqref{p.xi.xj} that
\begin{align*}
\Vert h\Vert^2
&=\sum_{i\in J} \sum_{j\in J} \left\langle p_i(B_{k_{i,1}},B_{k_{i,1}}^*,\dots,B_{k_{i,n_i}},B_{k_{i,n_i}}^*) p_j(B_{k_{j,1}},B_{k_{j,1}}^*,\dots,B_{k_{j,n_j}},B_{k_{j,n_j}}^*)^* b_j , b_i
\right\rangle
\\
&=\sum_{i\in J} \sum_{j\in J} \left\langle p_j(B_{k_{j,1}},B_{k_{j,1}}^*,\dots,B_{k_{j,n_j}},B_{k_{j,n_j}}^*)^* b_j , p_i(B_{k_{i,1}},B_{k_{i,1}}^*,\dots,B_{k_{i,n_i}},B_{k_{i,n_i}}^*)^* b_i
\right\rangle
\\
&=  \left\langle \sum_{j\in J} p_j(B_{k_{j,1}},B_{k_{j,1}}^*,\dots,B_{k_{j,n_j}},B_{k_{j,n_j}}^*)^* b_j ,
\sum_{i\in J} p_i(B_{k_{i,1}},B_{k_{i,1}}^*,\dots,B_{k_{i,n_i}},B_{k_{i,n_i}}^*)^* b_i
\right\rangle
\\
&=\Vert Wh\Vert^2.
\end{align*}
Hence, if $h,k\in \H_0$ are such that $\Vert h-k\Vert=0$, we get $\Vert Wh-Wk\Vert=\Vert W(h-k)\Vert=\Vert h-k\Vert=0$. Therefore, $W$ is a well-defined conjugate-linear isometry. Moreover, by its construction, $W$ is surjective. By density of $\H_0$ and $\K_0$ in $L_{\mathcal{A},A}$ and $L_{\mathcal{B},B}$, respectively, it follows that $W$ can be extended to a conjugate-linear surjective isometry from $L_{\mathcal{A},A}$ to $L_{\mathcal{B},B}$, which we also denote by $W$. Furthermore, it can be seen that $Wa_i=b_i$ for every $i\in I$.

Now, let $n\in\Zp$, $p\in\P{2n}$ and $\{k_l\}_{1\leq l\leq n}\subseteq K$, and let us check that
$$
T:=Wp(T_{k_1},T_{k_1}^*,\dots, T_{k_n},T_{k_n}^* )=S:=p(S_{k_1},S_{k_1}^*,\dots, S_{k_n},S_{k_n}^* )^*W.
$$
Let $m\in\Zp$, $q\in\P{2m}$, $\{t_r\}_{1\leq r\leq m}\subseteq K$ and $i\in I$, and put $h=q(A_{t_1},A_{t_1}^*,\dots, A_{t_m},A_{t_m}^*)a_i$. Then,
\begin{align*}
Th&=Wp(T_{k_1},T_{k_1}^*,\dots, T_{k_n},T_{k_n}^* )h\\
&=Wp(A_{k_1},A_{k_1}^*,\dots, A_{k_n},A_{k_n}^* )h\\
&=Wp(A_{k_1},A_{k_1}^*,\dots, A_{k_n},A_{k_n}^* )q(A_{t_1},A_{t_1}^*,\dots, A_{t_m},A_{t_m}^*)a_i\\
&=q(B_{t_1},B_{t_1}^*,\dots, B_{t_m},B_{t_m}^*)^*p(B_{k_1},B_{k_1}^*,\dots, B_{k_n},B_{k_n}^* )^*b_i\\
&=p(B_{k_1},B_{k_1}^*,\dots, B_{k_n},B_{k_n}^* )^* q(B_{t_1},B_{t_1}^*,\dots, B_{t_m},B_{t_m}^*)^*b_i\\
&= p(S_{k_1},S_{k_1}^*,\dots, S_{k_n},S_{k_n}^* )^*Wh=Sh.
\end{align*}
Since $n$, $q$, $\{k_l\}_{1\leq l\leq n}$ and $i$ were arbitrary, and taking into account that $T$ and $S$ are both conjugate-linear, we obtain $Th=Sh$ for every $h\in \H_0$. Consequently, the desired equality follows by density of $\H_0$ in $L_{\mathcal{A},A}$.
\end{proof}

The following lemma will be very useful throughout the paper.

\begin{lemma}\label{TN=M*T}
Let $N\in\mathcal{B}(\H)$ and $M\in\mathcal{B}(\K)$ be normal operators, and let $T:\H\to \K$ be a bounded linear (or conjugate-linear) operator. The following statements hold:
\begin{enumerate}[\rm (i)]
\item If $T$ is linear (resp. conjugate-linear) and $TN=MT$ (resp. $TN=M^*T$), then
$$
T E_{N}(\Delta)=E_{M}(\Delta)T\quad\mbox{for every Borel subset }\Delta\subseteq\C.
$$
\item If $T$ is linear (resp. conjugate-linear) and $TN=-MT$ (resp. $TN=-M^*T$), then
$$
T E_{N}(\Delta)=E_{M}(-\Delta)T\quad\mbox{for every Borel subset }\Delta\subseteq\C.
$$
\end{enumerate}
\end{lemma}

\begin{proof}
(i) Suppose first that $T$ is linear and $TN=MT$, and put
$$
L=
\begin{bmatrix}
N & 0\\
0 & M
\end{bmatrix}
\begin{array}{l}
\H\\
\K
\end{array}
\quad\mbox{and}\quad
A=
\begin{bmatrix}
0 & 0\\
T & 0
\end{bmatrix}
\begin{array}{l}
\H\\
\K
\end{array}.
$$
Then, $L$ and $A$ are bounded linear operators acting on the Hilbert space $\H\oplus \K$. Furthermore, one can verify that $AL=LA$. Since $L$ is normal, by the Fuglede Theorem, $AE_{L}(\Delta)=E_{L}(\Delta)A$ for every Borel subset $\Delta\subseteq\C$.

Fix a Borel subset $\Delta\subseteq\C$, and note that
$$
E_{L}(\Delta)=
\begin{bmatrix}
E_{N}(\Delta) & 0\\
0 & E_{M}(\Delta)
\end{bmatrix}
\begin{array}{l}
\H\\
\K
\end{array}.
$$
Therefore, the matrix equality $AE_{L}(\Delta)=E_{L}(\Delta)A$ implies $TE_{N}(\Delta)=E_{M}(\Delta)T$.

Assume now that $T$ is a conjugate-linear operator such that $TN=M^*T$, and let $\Delta\subseteq\C$ be a Borel subset. Since $M$ is normal, there exists a conjugation $C$ on $\K$ such that $CMC=M^*$. Hence, $TN=M^*T=(CMC)T$, and since $C^{-1}=C$, we obtain
\begin{equation}\label{(CT)N=M(CT)}
(CT)N=M(CT).
\end{equation}
Clearly, $CT:\H\to \K$ is a bounded linear operator. By our previous result, we have
\begin{equation}\label{T.conj.lin}
CTE_{N}(\Delta)=E_{M}(\Delta)CT.
\end{equation}
By \cite[p. 287, Lemma 8.7]{Conway.Func.An}, $E_{M}(\Delta)=\chi_{\Delta}(M)\in W^*(M)$, the von Neumann algebra generated by $M$. Hence, there exists a net $\{p_i\}_i$ in $\P{2}$ such that $p_i(M,M^*)$ converges to $E_{M}(\Delta)$ in the weak operator topology. One can easily check that $C p_i(M,M^*) C=p_i(M,M^*)^*$ for any $i$. Consequently, taking the limit in the weak operator topology, we get
$
CE_{M}(\Delta)C=E_{M}(\Delta)^*=E_{M}(\Delta).
$
Therefore, from \eqref{T.conj.lin}, we obtain
$$
TE_{N}(\Delta)=C^2TE_{N}(\Delta)=CE_{M}(\Delta)CT=E_{M}(\Delta)T,
$$
the desired equality.

\medskip

(ii) Suppose that $TN=-MT$ if $T$ is linear and that $TN=-M^*T$ if $T$ is conjugate-linear. Applying the first part of lemma to $N$ and $-M$, we get $TE_{N}(\Delta)=E_{-M}(\Delta)T$ for every Borel subset $\Delta\subseteq\C$. Since $E_{-M}(\Delta)=\chi_{\Delta}(-M)$ and $\chi_{\Delta}(-z)=\chi_{-\Delta}(z)$ for every $z\in\C$, we have
$$
TE_{N}(\Delta)=E_{-M}(\Delta)T=\chi_{\Delta}(-M)T=\chi_{-\Delta}(M)T=E_{M}(-\Delta)T.
$$
This completes the proof.
\end{proof}

\medskip

We are now in a position to prove Theorems \ref{interpo.Nk.Mk} and \ref{interpo.Nk.Csym}
 and \ref{interpo.Csym}.
 
\begin{proof}[Proof of Theorem \ref{interpo.Nk.Mk}]
(i)$\implies$(ii). Adopt the notations in Theorem \ref{interpo.Nk.Mk} (i). Since, for every $k\in K$, $CN_k=M_k^*C$ and $CM_k=N_k^*C$, it follows by Lemma \ref{TN=M*T} that for all $k\in K_0$ and $l\in K_1$, $CE_{M_k}(\Delta_k)=E_{N_k}(\Delta_k)C$ and $CE_{N_l}(D_l)=E_{M_l}(D_l)C$. Hence,
$$
\left(\prod_{k\in K_0} E_{N_k}(\Delta_k)\right) C=C\left(\prod_{k\in K_0} E_{M_k}(\Delta_k)\right)
\quad\mbox{and}\quad
\left(\prod_{l\in K_1} E_{M_l}(D_l)\right) C=C\left(\prod_{l\in K_1} E_{N_l}(D_l)\right).
$$
Thus, taking into account the equalities $Cx_i=y_i$ and $Cy_i=x_i$, we get
\begin{align*}
\left\langle \left(\prod_{k\in K_0} E_{N_k}(\Delta_k)\right) x_i, \left(\prod_{l\in K_1} E_{M_l}(D_l)\right)y_j \right\rangle
&=\left\langle \left(\prod_{k\in K_0} E_{N_k}(\Delta_k)\right) Cy_i, \left(\prod_{l\in K_1} E_{M_l}(D_l)\right)Cx_j \right\rangle\\
&=\left\langle C\left(\prod_{k\in K_0} E_{M_k}(\Delta_k)\right) y_i, C\left(\prod_{l\in K_1} E_{N_l}(D_l)\right)x_j \right\rangle\\
&=\left\langle \left(\prod_{l\in K_1} E_{N_l}(D_l)\right)x_j, \left(\prod_{k\in K_0} E_{M_k}(\Delta_k)\right) y_i \right\rangle.
\end{align*}
The other equality can be obtained in a similar way.

\medskip

(ii)$\implies$(i). Put $K_1=K_2=K$ and let $\widetilde{K}$ be the disjoint union of $K_1$ and $K_2$. Let $\mathcal{A}=\{ A_k \}_{k\in \widetilde{K}}$ and $\mathcal{B}=\{ B_k \}_{k\in \widetilde{K}}$ where
$$
A_k = \begin{cases}
  N_k &  \text{ if } k\in K_1; \\
  M_k  &  \text{ if } k\in K_2,
\end{cases}
\quad\mbox{and}\quad
B_k = \begin{cases}
  M_k  &  \text{ if } k\in K_1; \\
  N_k  &  \text{ if } k\in K_2.
\end{cases}
$$
Now, define $\widetilde{I}$ to be the disjoint union of $I_1$ and $I_2$ where $I_1=I_2=I$, and let $A=\{ a_i \}_{i\in \widetilde{I}}$ and $B=\{ b_i \}_{i\in \widetilde{I}}$ be such that
$$
a_i = \begin{cases}
  x_i  &  \text{ if } i\in I_1; \\
  y_i  &  \text{ if } i\in I_2,
\end{cases}
\quad\mbox{and}\quad
b_i = \begin{cases}
  y_i &  \text{ if } i\in I_1; \\
  x_i &  \text{ if } i\in I_2.
\end{cases}
$$
Since, for all subsets $\{k_l\}_{l=1}^4\subseteq K$ and Borel subsets $\{\Delta_{k_l}\}_{l=1}^4$ of $\C$, the orthogonal projections $E_{N_k}(\Delta_{k_1})$, $E_{N_k}(\Delta_{k_2})$, $E_{M_k}(\Delta_{k_3})$ and $E_{M_k}(\Delta_{k_4})$ are mutually commuting, one can easily see that $\mathcal{A}=\{ A_k \}_{k\in \widetilde{K}}$, $\mathcal{B}=\{ B_k \}_{k\in \widetilde{K}}$, $A$ and $B$ satisfy the conditions in Proposition \ref{Wisom}. Consequently, and since $L_{\mathcal{A},A}=L_{\mathcal{B},B}=L_{\mathcal{N}\cup\mathcal{M},X\cup Y}=\H$, there exists a conjugate-linear surjective isometry $C: \H\to \H$ such that
\begin{equation}\label{Ca_i=b_i}
Ca_i=b_i\quad\mbox{and}\quad Cp(A_{k_1},A_{k_1}^*,\dots, A_{k_n},A_{k_n}^* )=p(B_{k_1},B_{k_1}^*,\dots, B_{k_n},B_{k_n}^* )^*C
\end{equation}
for all $i\in I$, $n\in\Zp$, $p\in\P{2n}$ and $\{k_l\}_{1\leq l\leq n}\subseteq \widetilde{K}$. Furthermore, by definition of $\mathcal{A}=\{ A_k \}_{k\in \widetilde{K}}$, $\mathcal{B}=\{ B_k \}_{k\in \widetilde{K}}$, $A$ and $B$, the roles of  $A_{k_l}$ (resp. $a_i$) and $B_{k_l}$ (resp. $b_i$) can be interchanged in \eqref{Ca_i=b_i}; that is,
\begin{equation}\label{Cb_i=a_i}
Cb_i=a_i\quad\mbox{and}\quad Cp(B_{k_1},B_{k_1}^*,\dots, B_{k_n},B_{k_n}^* )=p(A_{k_1},A_{k_1}^*,\dots, A_{k_n},A_{k_n}^* )^*C.
\end{equation}
In particular, we have
\begin{equation*}
Cx_i=y_i \quad\mbox{and}\quad CN_k=M_k^*C\quad\mbox{for every $k\in K$}.
\end{equation*}
Now, we only need to show that $C$ is a conjugation. Let us first show that $C^2=I_{\H}$.  Since $\H=L_{\mathcal{A},A}$, it suffices to show that $C^2h=h$ for every vector $h$ of the form $h=p(A_{k_1},A_{k_1}^*,\dots, A_{k_n},A_{k_n}^* )a_i$ where $i\in I$, $n\in\Zp$, $p\in\P{2n}$  and $\{k_l\}_{1\leq l\leq n}\subseteq \widetilde{K}$. Let $h$ be such vector; then, \eqref{Ca_i=b_i} and \eqref{Cb_i=a_i} imply
\begin{align*}
C^2h=C^2p(A_{k_1},A_{k_1}^*,\dots, A_{k_n},A_{k_n}^* )a_i&=Cp(B_{k_1},B_{k_1}^*,\dots, B_{k_n},B_{k_n}^* )^*b_i\\
&=p(A_{k_1},A_{k_1}^*,\dots, A_{k_n},A_{k_n}^* )a_i=h.
\end{align*}
Finally, the isometric property in Definition \ref{def.conj} (ii) can be obtain from \eqref{conjugate.relation} since $C^{\#}=C^{-1}=C$.
\end{proof}

\medskip

\begin{proof}[Proof of Theorem \ref{interpo.Nk.Csym}]
The implication (i)$\implies$(ii) is obvious.

\medskip

(ii)$\implies$(iii). Let $V_0:\ker(V)^{\perp}\to\ran(V)$ be the invertible conjugate-linear isometry  given by $V_0h=Vh$. Define the conjugate-linear map $V^{\#}:\H\to \H$ by $V^{\#}h=V_0^{-1}h$ if $h\in\ran(V)$ and $V^{\#}h=0$ if $h\in\ran(V)^{\perp}$. Then, $V^{\#}V$ is the orthogonal projection onto $\ker(V)^{\perp}$. Moreover, using the polarization identity, it follows that $\langle Vh,k\rangle=\langle V^{\#}k,h\rangle$ for all $h\in \ker(V)^{\perp}$ and $k\in \ran(V)$. Thus, 
\begin{equation}\label{conjugate.relation}
\langle Vh,k\rangle=\langle V^{\#}k,h\rangle\quad\mbox{for all }h,k\in \H.
\end{equation}

Let $K_0\subseteq K$ be a finite subset, and let $\{\Delta_k\}_{k\in K_0}$  be Borel subsets of $\C$. In view of Lemma \ref{TN=M*T}, for all $k\in K_0$, we have
$VE_{N_k}(\Delta_k)=E_{N_k}(\Delta_k)V$. Hence, $V\left(\prod_{k\in K_0}E_{N_k}(\Delta_k)\right)=\left(\prod_{k\in K_0}E_{N_k}(\Delta_k)\right)V$. Put $P=\prod_{k\in K_0}E_{N_k}(\Delta_k)$. Thus,
\begin{equation}\label{P.xi.xj}
\begin{split}
\langle Py_j,y_i\rangle=\langle PVx_j,Vx_i\rangle &=\langle VPx_j,Vx_i\rangle=\langle V^{\#}Vx_i,P x_j\rangle\\
&=\langle x_i,P x_j\rangle=\langle P x_i,x_j\rangle,
\end{split}
\end{equation}
and similarly, we have
\begin{equation}\label{P.xi.yj}
\begin{split}
\langle Px_j,y_i\rangle=\langle PVy_j,Vx_i\rangle &=\langle VPy_j,Vx_i\rangle=\langle V^{\#}Vx_i,P y_j\rangle\\
&=\langle x_i,P y_j\rangle=\langle P x_i,y_j\rangle.
\end{split}
\end{equation}

\medskip

(iii)$\implies$(i). Assume that $M_k=N_k$ for every $k\in K$. For convenience, put $\l_0=L_{\mathcal{N},X\cup Y}$ and $\l_1=\l_0^{\perp}$. Then, $\l_0$ reduces $N_k$ for every $k\in K$; additionally, $x_i,y_i\in \l_0$ for all $i\in I$. It is not hard to see that $\left\{ {N_k}_{\vert \l_0} \right\}_{k\in K}$, $\left\{ {M_k}_{\vert \l_0} \right\}_{k\in K}$, $X$ and $Y$ satisfy assertion (ii) of Theorem \ref{interpo.Nk.Mk}. Thus, by the same theorem, there exists a conjugation $C_0$ on $\l_0$ such that $C_0{N_k}_{\vert \l_0}C_0=\left({N_k}_{\vert \l_0}\right)^*$ and $C_0x_i=y_i$ for all $k\in K$ and $i\in I$.

On the other hand, since the normal operators ${N_k}_{\vert \l_1}$ are mutually commuting, they generate an abelian von Neumann algebra in $\mathcal{B}(\l_1)$. Hence, according to \cite[p. 52, Lemma II.2.8]{Davidson}, there exists a self-adjoint operator $A\in\mathcal{B}(\l_1)$ such that ${N_k}_{\vert \l_1}\in W^*(A)$ for every $k\in K$. By self-adjointness of $A$, there exists a conjugation $C_1$ on $\l_1$ such that $C_1AC_1=A$. Thus, any polynomial in $A$ is $C_1$-symmetric; moreover, since the class of $C_1$-symmetric operators is closed in the weak operator topology, we get
\begin{equation}\label{CommonConj}
C_1{N_k}_{\vert \l_1}C_1=\left({N_k}_{\vert \l_1}\right)^*\quad\mbox{for every $k\in K$}.
\end{equation}
Therefore, one can readily see that $C=C_0\oplus C_1$ is a conjugation on $\H=\l_0\oplus \l_1$ that satisfies all desired equalities.
\end{proof}

\medskip

\begin{proof}[Proof of Theorem \ref{interpo.Csym}]
The implication (ii)$\implies$(iii) is obvious.

\medskip

(iii)$\implies$(iv). Let $\Delta\subseteq\C$ be a Borel subset and $K_0\subseteq K$ a finite set. Then, $V\left(\prod_{k\in K_0} E_{N_k}(\Delta)\right)=\left(\prod_{k\in K_0} E_{N_k}(\Delta)\right)V$. Noting that $V^*V$ is the orthogonal projection onto $\ker(V)^{\perp}$ and letting $P=\prod_{k\in K_0} E_{N_k}(\Delta)$, we get
$$
\Vert Py\Vert^2=\langle P y,y\rangle=\langle PVx,Vx\rangle=\langle Px,V^*Vx\rangle=\langle Px,x\rangle=\Vert Px\Vert^2.
$$

\medskip

The implication (iv)$\implies$(i) follows from  Theorem \ref{interpo.Nk.Csym} (iii)$\implies$(i). In fact, for any finite subset $K_0\subseteq K$ and Borel subsets $\{ \Delta_k\}_{k\in K_0}$ of $\C$, the orthogonal projections $\{ E_{N_k}(\Delta_k)\}_{k\in K_0}$ are mutually commuting. Thus, their product $P$ is also an orthogonal projection; which gives
$$
\langle Px,x\rangle=\langle P^2x,x\rangle=\langle Px,Px\rangle=\Vert Px\Vert^2=\Vert Py\Vert^2=\langle Py,y\rangle.
$$

\medskip

(i)$\implies$(v). In view of \eqref{CxoyC}, for any $k\in K$ and $\lambda\in\C$, we have
$$
C\left( N_k+ \lambda x\otimes y\right)C=CN_kC+\overline{\lambda}y\otimes x=N_k^*+\overline{\lambda}(x\otimes y)^*=\left( N_k+ \lambda x\otimes y\right)^*;
$$
that is, $N_k+ \lambda x\otimes y$ is $C$-symmetric.

\medskip

(v)$\implies$(ii). Since the class of $C$-symmetric operators forms a subspace of $\lh$, we obtain, for every $k\in K$, $CN_kC=N_k^*$ and $C(x\otimes y)C=(x\otimes y)^*=y\otimes x$. Given that $x$ and $y$ have the same norm, the latter equality implies the existence of a unimodular $\alpha\in \C$ such that $Cx=\alpha y$.

Applying Theorem \ref{interpo.Nk.Csym} (iii)$\implies$(i) to the particular case $\{ x_i\}_i=\{ y_i\}_i=\{x\}$, we obtain a conjugation $J$ on $\H$ such that $JN_kJ=N_k^*$ for every $k\in K$ and $Jx=x$. 
Put $U=\overline{\alpha}CJ$; then, since $C$ and $J$ are both surjective isometries and $\alpha$ is unimodular, the linear operator $U:\H\to \H$ must be unitary. Moreover, we have 
$$
UN_k=\overline{\alpha}CJN_k=\overline{\alpha}CN_k^*J=\overline{\alpha}N_kCJ=N_kU
$$
and $Ux=\overline{\alpha}CJx=\overline{\alpha}Cx=\overline{\alpha}\alpha y=y$. This completes the proof of the theorem.
\end{proof}

\section{Interpolation theorems for conjugations : C-skew symmetric versions}\label{interpolation.section.Cskew}

A conjugate-linear operator $V$ on $\H$ is said to be {\it anti-unitary} if it is a surjective isometry. Clearly, every conjugation on $\H$ is anti-unitary.

For two operators $A$ and $B$ on Hilbert spaces, we denote $A \cong B$ if there exists an isometric isomorphism $U$ between the underlying Hilbert spaces such that $UAU^{-1}=B$; or equivalently, $UA = BU$. It is worth mentioning that both complex symmetry and skew symmetry are preserved under the equivalence relation $\cong$.

For $k \in \{ 1, 2, \dots, \infty \}$, where $\infty = \aleph_0$, let $\H^{(k)}$ denote the direct sum of $\H$ with itself $k$ times. Accordingly, we denote by $T^{(k)} \in \mathcal{B}(\H^{(k)})$ the direct sum of $T \in \lh$.

\subsection{Main results}

\begin{theorem}\label{interpo.Cskew1}
The following statements are equivalent:
\begin{enumerate}[\rm (i)]
\item There exists a conjugation $C$ on $\H$ such that $CN_kC=-N_k^*$ and $Cx_i=y_i$ for all $k\in K$ and $i\in I$.

\item There exists an anti-unitary operator $V$ on $\H$ such that $VN_kV^{-1}=-N_k^*$, $Vx_i=y_i$ and $Vy_i=x_i$ for all $k\in K$ and $i\in I$.

\item There exists an anti-unitary operator $V_1$ on $\l_1:=\left(L_{\mathcal{N},X\cup Y}\right)^{\perp}$ such that 
$$
V_1{N_k}_{\vert \l_1}V_1^{-1}=-\left({N_k}_{\vert \l_1}\right)^*
\quad\mbox{for every $k\in K$},
$$
and for all finite subsets $K_0\subseteq K$, Borel subsets $\{\Delta_k\}_{k\in K_0}$ of $\C$,  and $i,j\in I$ we have
\begin{equation*}\label{Nk,-Nk*1}
\left\langle \left(\prod_{k\in K_0} E_{N_k}(\Delta_k)\right) x_i,x_j \right\rangle
=
\left\langle \left(\prod_{k\in K_0} E_{N_k}(-\Delta_k)\right) y_j,y_i \right\rangle
\end{equation*}
and
\begin{equation*}\label{Nk,-Nk*2}
\left\langle \left(\prod_{k\in K_0} E_{N_k}(\Delta_k)\right)x_i,y_j \right\rangle
=
\left\langle \left(\prod_{k\in K_0} E_{N_k}(-\Delta_k)\right) x_j,y_i \right\rangle.
\end{equation*}
\end{enumerate}
\end{theorem}

\medskip

Before stating the second main result of this section, we recall some facts about normal operators on separable complex Hilbert spaces. Given a finite measure with compact support $\mu$ on $\C$, denote by $M_{\mu}$ the operator defined on $L^2(\mu)$ by $[M_{\mu}f](z)=zf(z)$ for all $f\in L^2(\mu)$. According to spectral multiplicity theory (see, for instance, \cite[p. 298, Theorem 10.16]{Conway.Func.An}), each normal operator $N$ acting on $\H$ is unitarily equivalent to
\begin{equation}\label{multip.decom}
M_{\mu_{\infty}}^{(\infty)}\oplus\bigoplus_{1\leq i <\infty} M_{\mu_i}^{(i)}
\end{equation}
where $\mu_i$, $1\leq i\leq \infty$ (some of which may be zero), are mutually singular finite measures on Borel subsets of $\C$, each with compact support. Moreover, the previous decomposition is unique in the sense that if $M$ is another normal operator with corresponding measures $\nu_i$, $1\leq i\leq \infty$, then $N\cong M$ if and only if, for every $i\in\{1,2,\dots,\infty\}$,
\begin{equation}\label{uniq.dec}
\mu_i(\Delta)=0 \;\;\Longleftrightarrow\;\; \nu_i(\Delta)=0\quad\mbox{for every Borel subset }\Delta\subseteq\C.
\end{equation}

In \cite[Theorem 1.11]{Li.Zhu}, Li et al. proved that $N$ is skew-symmetric if and only if the sequence $\{\mu_i\}_{1\leq i\leq \infty}$ can be chosen so that, for every $i\in\{1,2,\dots,\infty\}$,
\begin{equation}\label{mu(-delta)}
\mu_i(\Delta)=\mu_i(-\Delta)\quad\mbox{for every Borel subset }\Delta\subseteq\C.
\end{equation}

\medskip

Given a normal operator $N\in\lh$, we write $m(N)<\infty$ if the measure $\mu_{\infty}$ in \eqref{multip.decom} is zero; in other words, $M_{\mu_{\infty}}^{(\infty)}$ can be omitted from the decomposition \eqref{multip.decom}. This definition makes sense due to the uniqueness property described in \eqref{uniq.dec}.

\begin{theorem}\label{interpo.Cskew2}
Let $N\in\lh$ be a skew-symmetric normal operator, $N_0$ be the restriction of $N$ to $L_{\{N\}, X\cup Y}$, and assume that  $m\left({N_{0}}_{\vert \ker(N_{0})^{\perp}}\right)<\infty$. The following statements are equivalent:
\begin{enumerate}[\rm (i)]
\item There exists a conjugation $C$ on $\H$ such that $CNC=-N^*$ and $Cx_i=y_i$ for every $i\in I$.

\item For every Borel subset $\Delta\subseteq\C$ and all $i,j\in I$, we have
$$
\langle E_{N}(\Delta)x_i,x_j \rangle=\langle E_{N}(-\Delta)y_j,y_i \rangle
\quad\mbox{and}\quad
\langle E_{N}(\Delta)x_i,y_j \rangle=\langle E_{N}(-\Delta)x_j,y_i \rangle.
$$
\end{enumerate}
In particular, if $I$ is finite, then $m\left({N_{0}}_{\vert \ker(N_{0})^{\perp}}\right)<\infty$ and the previous statements are equivalent.
\end{theorem}

\medskip

The proofs of Theorems \ref{interpo.Cskew1} and \ref{interpo.Cskew2} will be given at the end of this section.

\medskip

The reader can see from Theorem \ref{interpo.Cskew1} that the implication (i)$\implies$(ii), in the previous theorem, is true without the assumption that $m\left({N_{Z}}_{\vert \ker(N_{Z})^{\perp}}\right)<\infty$. However, the following example shows that this condition is not superfluous for the implication  (ii)$\implies$(i).

\begin{example}\label{counterexample}
Suppose that $\H=\K_1\oplus \K_2$, where $\K_1$ and $\K_2$ are infinite-dimensional separable complex Hilbert spaces, and let $\{e_i\}_{i\geq1}$ and $\{f_i\}_{i\geq1}$ be orthonormal bases of $\K_1$ and $\K_2$, respectively. For each $i\geq1$, put $x_i=e_i\oplus0$  $y_i= 0\oplus f_{i+1}$. Let $N$ be the normal operator on $\H$ given by $N=I_{\K_1}\oplus -I_{\K_2}$. Clearly, $N\cong -N$; hence, by Lemma \ref{cong.skewsym}, $N$ is skew-symmetric.

For convenience, put $\H_1=\vee\{ x_i: i\geq1\}$, $\H_2=\vee\{ y_i: i\geq1\}$ and $\H_3=\vee\{ 0\oplus f_1\}$. Then,  $\H=\H_1\oplus \H_2\oplus \H_3$ and $L_{\{N\},X\cup Y}=\H_1\oplus \H_2$. Furthermore,
$$
{N_{0}}_{\vert \ker(N_{0})^{\perp}}={N_{0}}
=
\begin{bmatrix}
I_{\H_1} & 0 \\
0 & -I_{\H_2}
\end{bmatrix}
\begin{array}{l}
\H_1\\
\H_2
\end{array}
\cong M_{\delta_{1}}^{(\infty)}\oplus M_{\delta_{-1}}^{(\infty)} \cong \left( M_{\delta_{1}}\oplus M_{\delta_{-1}} \right)^{(\infty)}
\cong M_{\mu}^{(\infty)}
$$
where $\mu=\delta_{1}+\delta_{-1}$ and $\delta_a$ denotes the Dirac measure, at $a\in\C$, on Borel subsets of $\C$. Thus, $m\left({N_{0}}_{\vert \ker(N_{0})^{\perp}}\right) \nless \infty$.

Let us check that $N$ satisfies the conditions in Theorem \ref{interpo.Cskew2} (ii). One can easily see that, for every Borel subset $\Delta\subseteq\C$, the spectral projections $E_N(\Delta)$ and $E_N(-\Delta)$ fall into one of the following possibilities:
\begin{enumerate}[\rm (a)]
\item $E_N(\Delta)=E_N(-\Delta)=0$;
\item $E_N(\Delta)=E_N(-\Delta)=I_{\H}$;
\item $E_N(\Delta)=I_{\K_1}\oplus 0$ and $E_N(-\Delta)=0\oplus I_{\K_2}$;
\item $E_N(\Delta)=0\oplus I_{\K_2}$ and $E_N(-\Delta)=I_{\K_1}\oplus 0$.
\end{enumerate}
In all cases, we have
$$
\langle E_{N}(\Delta)x_i,x_j \rangle=\langle E_{N}(-\Delta)y_j,y_i \rangle
\quad\mbox{and}\quad
\langle E_{N}(\Delta)x_i,y_j \rangle=\langle E_{N}(-\Delta)x_j,y_i \rangle.
$$

Now, let us show that there exists no conjugation $C$ on $\H$ satisfying the first assertion of Theorem \ref{interpo.Cskew2}. Suppose for contradiction that such a conjugation exists. Then, 
$$
C\H_1=C\vee\{ x_i: i\geq1\}=\vee\{ Cx_i: i\geq1\}=\vee\{ y_i: i\geq1\}=\H_2,
$$
and since $C^{-1}=C$, we also have $C\H_2=\H_1$, and hence $C(\H_1+\H_2)=\H_1+\H_2$. As $C$ is a surjective isometry, it follows that $C\H_3=C(\H_1+\H_2)^{\perp}=(\H_1+\H_2)^{\perp}=\H_3$. Therefore, $C$ takes the form
$$
C=
\begin{bmatrix}
0 & * & 0\\
* & 0 & 0\\
0 & 0 & C_3
\end{bmatrix}
\begin{array}{l}
\H_1\\
\H_2\\
\H_3
\end{array}
$$
where $C_3$ is a conjugation on $\H_3$. Finally, since $\H_3$ reduces $N$ and $CNC=-N^*$, we obtain
$-I_{\H_3}=C_3(-I_{\H_3})C_3=C_3N_{\vert \H_3}C_3=-N_{\vert \H_3}^*=I_{\H_3}$; a contradiction.
\end{example}

\medskip

When $X=\{ x\}$, $Y=\{ y\}$, and $\{x,y\}$ are linearly independent, we can add the following statement to Theorem \ref{interpo.Cskew1}.

\begin{theorem}\label{interpo.Cskew.x.y}
Let $\{x,y\}$ be linearly independent vectors in $\H$. The following statements are equivalent:
\begin{enumerate}[\rm (i)]
\item There exists a conjugation $C$ on $\H$ satisfying $CN_kC=-N_k^*$, for every $k\in K$, and $Cx=y$.
\item There exists a conjugation $C$ on $\H$ such that $N_k+\lambda (x\otimes x - y\otimes y)$ is $C$-skew symmetric for all $k\in K$ and $\lambda\in\C$.
\end{enumerate}
\end{theorem}

The proof of the previous theorem does not require preliminary results and can be presented here.

\begin{proof}[Proof of Theorem \ref{interpo.Cskew.x.y}]
(i)$\implies$(ii). Let $C$ be a conjugation satisfying (i). Then, for all $k\in K$ and $\lambda\in\C$, we have
\begin{align*}
C\big(N_k+\lambda (x\otimes x -y\otimes y) \big)C &=CN_kC+\overline{\lambda}\big((Cx)\otimes(Cx)-(Cy)\otimes(Cy)\big)\\
&=-N_k^*+\overline{\lambda}(y\otimes y - x\otimes x)\\
&=-\left(N_k+\lambda (x\otimes x -y\otimes y) \right)^*.
\end{align*}

(ii)$\implies$(i). For convenience, put $T=x\otimes x -y\otimes y$. By the fact that the class of $C$-skew symmetric operators is a subspace of $\lh$, we clearly have $CN_kC=-N_k^*$, for every $k\in K$, and $CTC=-T^*=-T$. Noting that $T$ is skew-symmetric, it must have a skew-symmetric matrix with respect to some orthonormal basis. In particular, since $T$ is a rank-two operator, its trace is well-defined and should be zero. Combining this with the fact that $T$ is self-adjoint, it follows that there exists orthonormal vectors $e_1,e_2\in \H$ and $\alpha\in \R$ such that
$$
\ran(T)=\ker(T-\alpha I_{\H})\oplus\ker(T+\alpha I_{\H}),
$$
$$\ker(T-\alpha I_{\H})=\vee\{e_1\}\quad\mbox{and}\quad\ker(T+\alpha I_{\H})=\vee\{e_2\}.
$$
As $CTC=-T$, we get $C(T+\alpha I_{\H})C=-(T-\alpha I_{\H})$; hence
\begin{equation}\label{Cker}
C(\vee\{e_1\})=C \ker(T-\alpha I_{\H})=\ker(T+\alpha I_{\H})=\vee\{ e_2\}.
\end{equation}
Note that $x$ and $y$ have the same norm. In fact, $\langle x,x\rangle-\langle y,y\rangle$ is the trace of $T$. Hence, using \cite[Theorem 2.1]{Zhu.Li}, we obtain a conjugation $J$ on $\H$ satisfying $Jx=y$; one can check that $JTJ=-T$. The argument in \eqref{Cker} shows that $J(\vee\{e_1\})=\vee\{e_2\}$. Then, since $C$ and $J$ are isometries, there exist unimodular scalars $a,b\in\C$ such that $Ce_1=ae_2$ and $Je_1=be_2$. Furthermore, since $C^{-1}=C$ and $J^{-1}=J$, we get $Ce_2=a e_1$ and $Je_2=b e_1$. Noting that $C$ and $J$ are both conjugate-linear, we get $Ch=a\overline{b}Jh$ for every $h\in \vee\{e_1,e_2\}$. In particular, since $x\in\ran(T)=\vee\{e_1,e_2\}$, we obtain $Cx=a\overline{b}Jx=a\overline{b}y$. Put $\widetilde{C}=\overline{a}bC$; one can easily check that $\widetilde{C}$ is a conjugation that satisfies $\widetilde{C}N_k\widetilde{C}=-N_k^*$, for every $k\in K$, and $\widetilde{C}x=y$. This ends the proof of the theorem.
\end{proof}

\medskip

Next, we shall apply the main results of this section to characterize skew symmetry of operators of the form
\begin{equation}\label{ssh.dense}
\sum_{j\in J} r_j(e_j\otimes e_j -f_j\otimes f_j) + {\rm i}B
\end{equation}
where the set $\{e_j,f_j : j\in J\}\subset\H$ is orthonormal, $\{r_j\}_{j\in J}$ are distinct positive scalars and $B\in\lh$ is self-adjoint.

Before proceeding, it is worth noting that skew symmetric operators having the previous form are dense in $\ssh$. Let us prove this claim. Let $T\in\lh$ be a $C$-skew symmetric operator for some conjugation $C$ on $\H$. A quick examination of the cartesian decomposition of $T$ shows that we can write $T=A+{\rm i}B$ where $A$ and $B$ are $C$-skew symmetric self-adjoint operators. Using \cite[Lemma 3.2]{Bu.Zhu}, we obtain a sequence of self-adjoint diagonalizable $C$-skew symmetric operators $A_n$ that converges to $A$, with $\dim\ker(A_n - r I_{\H})=1$ whenever  $r$ is an eigenvalue of $A_n$. Now, by skew-symmetry of $A_n$, we have $\dim\ker(A_n - r I_{\H})=\dim\ker(A_n + r I_{\H})$ for all real scalars $r$. Consequently, $T$ can be approximated by operators of the form \eqref{ssh.dense}.

\begin{corollary}\label{skew.sym.criteria}
Let $T\in\lh$ be an operator of the form \eqref{ssh.dense} and denote by $B_1$ the restriction of $B$ to $\left(L_{\{B\},\{e_j,f_j : j\in J\}}\right)^{\perp}$. Then, $T$ is skew-symmetric if and only if $B_1\cong -B_1$ and there exists unimodular scalars $\lambda_j$, for $j\in J$, such that
\begin{equation}\label{skew.sym.dense}
\lambda_i\left\langle  E_{B}(\Delta) e_i,e_j \right\rangle
=
\lambda_j\left\langle  E_{B}(-\Delta) f_j,f_i \right\rangle
\;\mbox{and}\;
\lambda_i\left\langle  E_{B}(\Delta) e_i,f_j \right\rangle
=
\lambda_j\left\langle  E_{B}(-\Delta) e_j,f_i \right\rangle
\end{equation}
for every Borel subset $\Delta$ of $\C$.
\end{corollary}

\begin{proof}
Assume that \eqref{skew.sym.dense} holds for every Borel subset $\Delta$ of $\C$. Since $B_1\cong -B_1$, Lemma \ref{cong.skewsym} implies that $VB_1V^{-1}=-B_1^*$ for some anti-unitary operator $V$. Then, applying Theorem \ref{interpo.Cskew1} to the self-adjoint operator $B$ and the orthonormal subsets $\{\lambda_j e_j\}_{j\in J}$ and $\{ f_j\}_{j\in J}$, we obtain a conjugation $C$ on $\H$ such that $CBC=-B$ and $C(\lambda e_j)=f_j$ for every $j\in J$. One can check that $C(e_j\otimes e_j -f_j\otimes f_j)C=-(e_j\otimes e_j -f_j\otimes f_j)=-(e_j\otimes e_j -f_j\otimes f_j)^*$ for every $j\in J$. Thus, $T$ is $C$-skew symmetric as the sum of two $C$-skew symmetric ones. This proves the sufficiency.

\medskip

Suppose that $T$ is $C$-skew symmetric for some conjugation $C$ on $\H$, and denote $A=\sum_{j\in J} r_j(e_j\otimes e_j -f_j\otimes f_j)$. Then, $A$ and $B$ are both  $C$-skew symmetric. Moreover, for every $j\in J$, $\ker(A-r_j I_{\H})=\vee\{e_j\}$ and $\ker(A+r_j I_{\H})=\vee\{f_i\}$. As in \eqref{Cker}, we have that $Ce_j\in\spn\{f_j\}$ for every $j\in J$. Hence, there exist unimodular scalars $\lambda_j$, for $j\in J$, so that $C(\lambda_j e_j)=f_j$ for every $j\in J$. The necessity follows now by combining Theorem \ref{interpo.Cskew1} and Lemma \ref{cong.skewsym}.
\end{proof}

\begin{remark}
\begin{enumerate}[\rm (i)]
\item If the operator $T$ has rather the form
$$
T=B + {\rm i}\sum_{j\in J} r_j(e_j\otimes e_j -f_j\otimes f_j)
$$
with $\{e_j,f_j\}_{j\in J}$, $\{r_j\}_{j\in J}$ and $B$ being as described in \eqref{ssh.dense}, 
then, taking into account that skew symmetry is invariant under scalar multiplication, the equivalence in the previous corollary still holds true.  
\item If $J$ is finite, then using Theorem \ref{interpo.Cskew2}, one can make minor changes in the proof of Corollary \ref{skew.sym.criteria} to show that the condition ``$B_1\cong -B_1$'' can be replaced with the simpler one ``$B\cong -B$''.
\end{enumerate}

\end{remark}


The remainder of this section is devoted to proving Theorems \ref{interpo.Cskew1} and \ref{interpo.Cskew2}.

\subsection{Proof of Theorems \ref{interpo.Cskew1} and \ref{interpo.Cskew2}}

The following lemma provides conditions under which a family of mutually commuting normal operators have a common orthonormal basis in which their corresponding matrices are skew-symmetric. This result is inspired by \cite[Theorem 1.10]{Li.Zhu}.

\begin{lemma}\label{cong.skewsym}
Let $\{N_k\}_{k\in K}\subset\lh$ be mutually commuting normal operators. The following statements are equivalent:
\begin{enumerate}[\rm (i)]
\item There exists a conjugation $C$ on $\H$ such that
$$
CN_kC=-N_k^*\quad \mbox{for every $k\in K$}.
$$
\item There exists an anti-unitary operator $V$ on $\H$ such that $$
VN_kV^{-1}=-N_k^*\quad \mbox{for every $k\in K$}.
$$
\item There exists a unitary operator $U$ on $\H$ such that
$$
UN_kU^*=-N_k\quad \mbox{for every $k\in K$}.
$$
\end{enumerate}
\end{lemma}

\begin{proof}
The implication (i)$\implies$(ii) is obvious.

\medskip

(ii)$\implies$(iii). Let $V$ be an anti-unitary operator on $\H$ such that $VN_kV^{-1}=-N_k^*$ for every $k\in K$. As in \eqref{CommonConj}, there exists a conjugation $J$ on $\H$ that satisfies $JN_kJ=N_k^*$ for every $k\in K$. Hence, we have $VN_kV^{-1}=-JN_kJ$, and so, since $J^{-1}=J$, we get
\begin{equation}\label{JVN_k}
(JV)N_k(V^{-1}J)=-J^2N_kJ^2=-N_k \quad \mbox{for every $k\in K$}.
\end{equation}
Put $U=JV$; then $U\in\lh$ is unitary and $U^*=U^{-1}=(JV)^{-1}=V^{-1}J^{-1}=V^{-1}J$. Therefore, for any $k\in K$, $UN_kU^{*}=-N_k$. 

\medskip

(iii)$\implies$(i). It is easy to see that we may assume without loss of generality that $K$ is infinite. Let $U\in\lh$ be a unitary operator satisfying $UN_k=-N_kU$ for every $k\in K$. First, let us reduce the problem to the case where $\bigcap_{k\in K}\ker(N_k)=\{0\}$. Denote this subspace by $\l_{\infty}$. One can easily see that $\l_{\infty}$ reduces $U$, and that each $N_k$ has the form $N_k=M_k\oplus 0$ with respect to the decomposition $\H=\l_{\infty}^{\perp}\oplus \l_{\infty}$. Clearly, $\bigcap_{k\in K}\ker(M_k)=\{0\}$; hence, if a conjugation $C$ on $\l_{\infty}^{\perp}$ exists such that $CM_kC=-M_k^*$ for each $k\in K$, then one can see that the map $C\oplus C_{\infty}$, where $C_{\infty}$ is an arbitrary conjugation on $\l_{\infty}$, is a conjugation on $\H$ that satisfies all the desired equalities.

Note that we can further assume that there exists a countable subset $\{N_{k_p}\}_{p\in\Zp}$ of $\{N_k\}_{k\in K}$ such that $\bigcap_{p\in\Zp}\ker(N_{k_p})=\{0\}$. In fact, since $\H=\bigvee\left( \bigcup_{k\in K} \ker(N_k)^{\perp} \right)$, one can use the separability of $\H$ to show that there exists a countable subset $\{k_p\}_{p\in\Zp}\subseteq K$ so that $\H=\bigvee\left( \bigcup_{p\in\Zp} \ker(N_{k_p})^{\perp} \right)$.

Fix a Borel subset $\Delta_0\subset\C$ such that $\C\setminus\Delta_0=-\Delta_0\cup\{0\}$. By Lemma \ref{TN=M*T}, we have
$
UE_{N_{k_1}}(\Delta_0)=E_{N_{k_1}}(-\Delta_0)U
$
and
$
UE_{N_{k_1}}(-\Delta_0)=E_{N_{k_1}}(\Delta_0)U.
$
Consequently,
\begin{equation}\label{Uran}
U\ran(E_{N_{k_1}}(\Delta_0))=\ran(E_{N_{k_1}}(-\Delta_0))
\;\;\mbox{and}\;\;
U\ran(E_{N_{k_1}}(-\Delta_0))=\ran(E_{N_{k_1}}(\Delta_0)).
\end{equation}

For convenience, put $\H_{+}=\ran(E_{N_{k_1}}(\Delta_0))$, $\H_{-}=\ran(E_{N_{k_1}}(-\Delta_0))$ and $\l_1=\ran(E_{N_{k_1}}(\{0\}))=\ker(N_{k_1})$. Given that $\C=\Delta_0 \cup (-\Delta_0) \cup \{0\}$ is a partition, it follows from \eqref{Uran} that $U$ has the form
$$
U=\begin{bmatrix}
0 & * & 0\\
W & 0 & 0\\
0 & 0 & U_1
\end{bmatrix}
\begin{array}{l}
\H_{+}\\
\H_{-}\\
\l_1
\end{array}
$$
where $W$ and $U_1$ are isometric isomorphisms between Hilbert spaces. Additionally, since for every $k\in K$,
$$
N_kE_{N_{k_1}}(\Delta_0)=E_{N_{k_1}}(\Delta_0)N_k
\quad\mbox{and}\quad
N_kE_{N_{k_1}}(-\Delta_0)=E_{N_{k_1}}(-\Delta_0)N_k,
$$
we can also write
$$
N_k=\begin{bmatrix}
N_{k,+} & 0 & 0\\
0 & N_{k,-} & 0\\
0 & 0 & R_{k,1}
\end{bmatrix}
\begin{array}{l}
\H_{+}\\
\H_{-}\\
\l_1
\end{array}
$$
where $N_{k,+}$, $N_{k,-}$ and $R_{k,1}$, $k\in K$, are normal operators such that the $R_{k,1}$'s are mutually commuting.

Since $UN_k=-N_kU$, simple matrix calculations show that
\begin{equation}
W N_{k,+} =-N_{k,-} W
\quad\mbox{and}\quad
U_1R_{k,1}=-R_{k,1}U_1
\quad\mbox{for every $k\in K$}.
\end{equation}
Put $\M_1=\H_+\oplus \H_-$ and let $V:\M_1\to \H_-\oplus \H_-$ be the unitary operator given by
$$
\begin{array}{rcl}
V: \H_+\oplus \H_-&\to& \H_-\oplus \H_-\\
(x,y) &\mapsto & (Wx,y).
\end{array}
$$
Then, it is easy to see that $V{N_{k}}_{\vert \M_1}V^{-1}=(-N_{k,-})\oplus N_{k,-}$ for every $k\in K$. Since the normal operators $\{ N_{k,-} \}_{k\in K}$ are mutually commuting, as in \eqref{CommonConj}, there exists a conjugation $J$ on $\H_-$ such that $JN_{k,-}J=\left(N_{k,-}\right)^*$ for every $k\in K$. Put
$
J_1=\begin{bmatrix}
0 & J\\
J & 0\\
\end{bmatrix}
\begin{array}{l}
\H_-\\
\H_-
\end{array}.
$
It is easy to see that $J_1$ is a conjugation on $\H_-\oplus \H_-$ satisfying
$$
J_1\left( (-N_{k,-})\oplus N_{k,-} \right)J_1=-\left((-N_{k,-})\oplus N_{k,-}\right)^*
\quad\mbox{for every $k\in K$}.
$$
Hence, $C_1:=V^{-1}J_1V$ is a conjugation on $\M_1$ that satisfies $C_1{N_k}_{\vert \M_1}C_1=-\left({N_k}_{\vert \M_1}\right)^*$ for every $k\in K$.

Since $\{R_{k,1}\}_{k\in K}$ are mutually commuting and $U_1R_{k,1}U_1^*=-R_{k,1}$ for every $k\in K$, by applying the preceding argument to $R_{k_2,1}\in\mathcal{B}(\l_1)$, we obtain subspaces $\M_2$ and $\l_2$ of $\l_1$, with $\l_1=\M_2\oplus \l_2$, a conjugation $C_2$ on $\M_2$ such that
$$
C_2{N_k}_{\vert \M_2}C_2=C_2{R_{k,1}}_{\vert \M_2}C_2=-\left( {R_{k,1}}_{\vert \M_2} \right)^*=-\left({N_k}_{\vert \M_2}\right)^*\quad\mbox{for every $k\in K$},
$$
a unitary operator $U_2$ on $\l_2$, and a set of mutually commuting normal operators $\{R_{k,2}\}_{k\in K}$ on $\l_2$ such that ${N_k}_{\vert \l_2}=R_{k,2}$ and $U_2R_{k,2}U_2^*=-\left(R_{k,2}\right)^*$. This process produces a sequence $\M_i$, $i\geq 1$, of reducing subspaces of $N_k$, for $k\in K$, and conjugations $C_i$ on $\M_i$ such that ${N_k}_{\vert \M_i}$ is $C_i$-skew symmetric for all $i\in\Zp$ and $k\in K$.

For each $i\geq 1$, it is easy to see that
$
\left(\sum_{j=1}^i \M_j\right)^{\perp}=\l_i=\bigcap_{j=1}^i\ker(N_{k_j}).
$
Hence,
$$
\left( \bigoplus_{i=1}^{\infty} \M_i \right)^{\perp}
=\bigcap_{i=1}^{\infty}\left(\sum_{j=1}^i \M_j\right)^{\perp}
=\bigcap_{i=1}^{\infty}\ker(N_{k_i})=\{0\}.
$$
Consequently, $\H=\bigoplus_{i=1}^{\infty} \M_i$; moreover, one can check that $C:=\bigoplus_{i=1}^{\infty} C_i$ is a conjugation on $\H$ satisfying $CN_kC=-N_k^*$ for every $k\in K$.
\end{proof}

\medskip

Note that every bounded linear operator $T$ acting on a Hilbert space can be written as $T=N\oplus A$ where $N$ is normal and $A$ is an operator that has no reducing subspace for which the restriction of $A$ is normal (of course, either $N$ or $A$ could be absent). In fact,
$N$ is the restriction of $T$ to the reducing subspace
\begin{equation}\label{normalpart}
\H_{\rm nor}=\bigcap_{1\leq n,m<\infty}\ker({T^*}^nT^m-T^m{T^*}^n).
\end{equation}

The operator $N$ is called the {\it normal part} of $T$, while $A$ is called the {\it abnormal part} of $T$. In this case, we denote $T_{\rm nor}=N$ and $T_{\rm abnor}=A$.

Using \eqref{normalpart}, one can easily see that if $M$ is a normal operator, then
\begin{equation}\label{nor.abnor}
(M\oplus T)_{\rm nor}=M\oplus T_{\rm nor}
\quad\mbox{and}\quad
(M\oplus T)_{\rm abnor}=T_{\rm abnor}.
\end{equation}

The following lemma is an immediate consequence of \cite[Proposition 3.10]{Zhu14} and \eqref{nor.abnor}.

\begin{lemma}\label{Tnor.Tabnor}
Let $M\in\lh$ be normal and $T\in\lk$ such that $M\oplus T$ is skew-symmetric. Then, $M\oplus T_{\rm nor}$ and $T_{\rm abnor}$ are skew-symmetric.
\end{lemma}

\begin{lemma}\label{skew.D.-D}
Let $N\in\lh$ be a skew-symmetric normal operator and $\Delta\subseteq\C$ a Borel subset such $\Delta=-\Delta$. Then, the restriction of $N$ to the reducing subspace $\ran(E_N(\Delta))$ is skew-symmetric.
\end{lemma}

\begin{proof}
Let $C$ be a conjugation on $\H$ such that $CNC=-N^*$. By Lemma \ref{TN=M*T}, we have $CE_N(\Delta)=E_N(-\Delta)C=E_N(\Delta)C$; hence, $C\ran(E_N(\Delta))=\ran(E_N(\Delta))$. Let $C_0$ and $N_0$ denote the restrictions of $C$ and $N$ to $\ran(E_N(\Delta))$, respectively. Then, $C_0$ is a conjugation on $E_N(\Delta)$. Moreover, it follows immediately from the equality $CNC=-N^*$ that $C_0N_{0}C_0=-N_{0}^*$.
\end{proof}

\medskip

In \cite[Lemma 3.2]{GUO.JI.ZHU} and \cite[Corollary 3.4]{Amara.Oudghiri2}, the authors proved that if $N$ is normal, then a direct sum $N\oplus T$ is complex symmetric (resp. conjugate-normal) if and only if $T$ is complex-symmetric  (resp. conjugate-normal). However, this property does not hold for skew-symmetry even if $N$ is skew-symmetric. The next result provides a complete description of normal operators $N$ that satisfy this property.

\begin{theorem}\label{directsumskew}
Let $N\in\lh$ be normal. The followings statements are equivalent:
\begin{enumerate}[\rm (i)]
\item For every separable complex Hilbert space $\K$ and every $T\in\mathcal{B}(\K)$,
$$
N\oplus T\mbox{ is skew-symmetric}\quad\Longleftrightarrow\quad  T\mbox{ is skew-symmetric.}
$$ 
\item $N$ is skew-symmetric and $m(N_{\vert\ker(N)^{\perp}})<\infty$.
\end{enumerate}
\end{theorem}

\begin{proof}
(i)$\implies$(ii). Suppose statement (i) holds. Taking $T=0$, the zero operator on any Hilbert space $\K$, we obtain that $N\oplus 0$ is skew-symmetric. By  \cite[Lemma 2.1]{Li.Zhu}, it follows that $N_{\vert(\ker N)^{\perp}}$ is skew-symmetric. Therefore, $N=N_{\vert(\ker N)^{\perp}}\oplus 0$ is skew-symmetric.

For convenience, put $N_1=N_{\vert\ker(N)^{\perp}}$, and let $N_1\cong M_{\mu_{\infty}}^{(\infty)}\oplus\bigoplus_{1\leq i <\infty} M_{\mu_i}^{(i)}$ be the canonical decomposition in \eqref{multip.decom} of $N_1$.

We need to show that $\mu_{\infty}=0$. Assume, for the sake of contradiction, that $\mu_{\infty}\neq 0$. Hence, $M_{\mu_{\infty}}$ is not absent and is not zero due to the injectivity of $N_1$. Since $\sigma(M_{\mu_{\infty}})\neq\{0\}$, using the spectral theorem, we can write $M_{\mu_{\infty}}=M\oplus L$, were $L$ is a normal operator with $\sigma(L)\neq -\sigma(L)$; in particular, $L\ncong -L$ and thus not skew-symmetric by Lemma \ref{cong.skewsym}.

Let us verify that $N\oplus L$ is skew symmetric. We have
\begin{align*}
N_1\oplus L \cong M_{\mu_{\infty}}^{(\infty)}\oplus\bigoplus_{1\leq i <\infty} M_{\mu_i}^{(i)}\oplus L 
&\cong (M\oplus L)^{(\infty)}\oplus L\oplus\bigoplus_{1\leq i <\infty} M_{\mu_i}^{(i)}\\
&\cong M^{(\infty)}\oplus L^{(\infty)}\oplus L \oplus \bigoplus_{1\leq i <\infty} M_{\mu_i}^{(i)}\\
&\cong M^{(\infty)}\oplus L^{(\infty)}\oplus  \bigoplus_{1\leq i <\infty} M_{\mu_i}^{(i)}\\
&\cong N_1.
\end{align*}
Thus, $N\oplus L\cong N$; in particular, $N\oplus L$ is skew-symmetric; a contradiction. Therefore, $m(N_1)<\infty$.

\medskip

(ii)$\implies$(i). Since $N$ is skew-symmetric, the reverse implication in (i) is always true. Fix a bounded linear operator $T$, assume that $N\oplus T$ is skew-symmetric, and let us show that $T$ is skew-symmetric.

In view of Lemma \ref{Tnor.Tabnor} and \cite[Lemma 2.1]{Li.Zhu}, we can assume without loss of generality that $T$ is normal and $N=N_{\vert \ker(N)^{\perp}}$. Moreover, by  \eqref{multip.decom} and \eqref{mu(-delta)}, we may directly assume that
$$
N=\bigoplus_{1\leq i <\infty} M_{\mu_i}^{(i)}
\quad\mbox{and}\quad
\H=\bigoplus_{1\leq i<\infty}L^2(\mu_i)^{(i)}
$$
where the $\mu_i$'s are mutually singular and satisfy $\mu_i(\Delta)=\mu_i(-\Delta)$ for every Borel subset $\Delta\subseteq\C$. Since the measures $\mu_i$'s are mutually singular and satisfy \eqref{mu(-delta)}, it is elementary to see that there exist mutually disjoint Borel subsets $\{ \Delta_i\}_{i\in\Zp}$ of $\C$ such that for all $i,j\in \Zp$ with $i\neq j$,
\begin{equation}\label{singularsubsets}
\Delta_i=-\Delta_i,\quad \mu_i(\Delta)=\mu_i(\Delta\cap\Delta_i),\quad \mu_i(\Delta_j)=0\;\;\mbox{ for every Borel }\Delta\subseteq\C. 
\end{equation}
Let $\Delta_0=\C\setminus\cup_{i\in \Zp} \Delta_i$; clearly, $\Delta_0=-\Delta_0$. For convenience, put $N_i=M_{\mu_i}^{(i)}$, $\H_i=L^2(\mu_i)^{(i)}$ for $i\in \Zp$, and let $T=T_0\oplus\bigoplus_{i\in \Zp} T_i$ be the decomposition of $T$ with respect to the orthogonal decomposition $\K=\ran(E_T(\Delta_0))\oplus\bigoplus_{i\in \Zp}\ran(E_T(\Delta_i))$. For $0\leq i <\infty$, put $\K_i=\ran(E_T(\Delta_i))$.

Since by  \eqref{singularsubsets}, $\mu_i(\Delta_0)=\mu_i(\Delta_j)=0$ all $i,j\in \Zp$ with $i\neq j$, it follows that $E_{N_i}(\Delta_0)=E_{N_i}(\Delta_j)=0$. Therefore,
\begin{equation*}
\ran(E_{N\oplus T}(\Delta_0))=\K_0
\quad\mbox{and}\quad
\ran(E_{N\oplus T}(\Delta_i))=\H_i\oplus \K_i \;\;\mbox{for every }i\in \Zp.
\end{equation*}
It follows that
$$
(N\oplus T)_{\vert \ran(E_{N\oplus T}(\Delta_0))}=T_0
\quad\mbox{and}\quad
(N\oplus T)_{\vert\ran(E_{N\oplus T}(\Delta_i))}=N_i\oplus T_i\;\;\mbox{for every }i\in \Zp.
$$
Consequently, by Lemma \ref{skew.D.-D}, $T_0$ and $N_i\oplus T_i$, for $i\in \Zp$, are skew-symmetric operators. In particular, by Lemma \ref{cong.skewsym}, we have
\begin{equation}\label{eq2}
N_i\oplus T_i\cong -N_i\oplus -T_i \quad\mbox{for every }i\in \Zp.
\end{equation}

Fix $i\in \Zp$, and let us prove that $T_i$ is skew-symmetric. Given $\mu_i(\Delta)=\mu_i(-\Delta)$ for every Borel subset $\Delta\subseteq\C$,  \cite[Theorem 1.11]{Li.Zhu} implies that $N_i$ is skew-symmetric; hence, $N_i\cong-N_i$ by Lemma \ref{cong.skewsym}. Combining this with \eqref{eq2}, we obtain $N_i\oplus T_i\cong N_i\oplus -T_i$. Since $M_{\mu_i}$ is a star-cyclic normal operator (\cite[p. 269, Theorem 3.4]{Conway.Func.An}), applying \cite[p. 295, Proposition 10.6]{Conway.Func.An} $i$ times to $M_{\mu_i}$, we conclude that $T_i\cong -T_i$, and so $T_i$ is skew-symmetric by Lemma \ref{cong.skewsym}. 

Thus, $T=T_0\oplus\bigoplus_{i\in\Zp} T_i$ is skew-symmetric as a direct sum of skew-symmetric operators.
\end{proof}

\begin{lemma}\label{twostarcyclics}
Let $N\in\lh$ be normal, and suppose that there exist reducing subspaces $\{\H_i\}_{1\leq i \leq r}$, $r<\infty$, of $N$ such that $\H=\vee(\sum_{i=1}^r \H_i)$ and $N_{\vert \H_i}$ is star-cyclic for each $i\in\{1,2,\dots,r\}$. Then, for every reducing subspace $L$ of $N$, we have  $m(N_{\vert L})<\infty$.
\end{lemma}

\begin{proof}
First, let us show that $N$ is a finite direct sum of star-cyclic normal operators.
We will proceed by induction on $r$. The result is clear when $r=1$. Assume the lemma holds for some integer $r\geq 1$. Let $\H_1,\H_2,\dots \H_{r+1}$ be $r+1$ reducing subspaces of $N$ such that $\H=\vee(\sum_{i=1}^{r+1} \H_i)$ and $N_{\vert \H_i}$ is star cyclic for $i\in\{1,2,\dots,r+1\}$. Let $\K=\vee(\sum_{i=1}^r \H_i)$ and $P\in\lh$ be the orthogonal projection onto $\K$. Then, $\K$ reduces $N$; moreover, by the induction hypothesis, $N_{\vert \K}$ is a finite direct sum of star-cyclic operators.

By star-cyclicity of $N_{\vert \H_{r+1}}$, there exists $h\in \H$ such that $\H_{r+1}=L_{\{N\},\{h\}}$. To complete the proof, it is clear that it suffices to show that $\H=\K\oplus L_{\{N\},\{(I-P)h\}}$. Note that $\K$ and $L_{\{N\},\{(I-P)h\}}$ are orthogonal. In fact, since $\K$ reduces $N$, its orthogonal projection $P$, and hence $(I-P)$, commutes with $N$ and $N^*$. Consequently, 
\begin{align*}
L_{\{N\},\{(I-P)h\}}&=\bigvee\big\{ p(N,N^*)(I-P)h : p\in\P{2}  \big\}\\
&=(I-P)\bigvee\big\{ p(N,N^*)h : p\in\P{2}  \big\}\subseteq \K^{\perp}.
\end{align*}
It remains to prove that $\H=\K + L_{\{N\},\{(I-P)h\}}$; we need only to show that $\H_{r+1}\subseteq \K+L_{\{N\},\{(I-P)h\}}$. For $p\in\P{2}$, we have
$$
p(N,N^*)h=p(N,N^*)Ph+p(N,N^*)(I-P)h.
$$
Since $Ph\in \K$ and $\K$ reduces $p(N,N^*)$, we have $p(N,N^*)Ph\in \K$. Hence, $p(N,N^*)h\in \K + L_{\{N\},\{(I-P)h\}}$, and thus $\H_{r+1}\subseteq \K+L_{\{N\},\{(I-P)h\}}$.

\medskip

Now, let $L$ be a reducing subspace of $N$ and assume, for contradiction, that $m(N_{\vert L})\nless\infty$. According to \eqref{multip.decom}, $N_{\vert L}$ has a reducing subspace for which its restriction is unitarily equivalent to $M_{\mu}^{(\infty)}$ for some finite non-zero compactly supported measure $\mu$ on $\C$. Thus, we can write $N\cong N_0\oplus M_{\mu}^{(\infty)}$ for some normal operator $N_0$. This implies
$$
N\oplus M_{\mu}^{(2)}\cong N_0 \oplus M_{\mu}^{(\infty)}\oplus M_{\mu}^{(2)}\cong N_0 \oplus M_{\mu}^{(\infty)}\oplus M_{\mu}\cong N\oplus M_{\mu}.
$$
Since $N$ is a direct sum of $n$ star-cyclic normal operators $M_1, M_2,\dots, M_n$ with $n<\infty$, applying \cite[p. 295, Proposition 10.6]{Conway.Func.An} to each $M_i$, we conclude $M_{\mu}\cong M_{\mu}^{(2)}$; which contradicts \eqref{uniq.dec}.
\end{proof}

We are now ready to prove Theorems \ref{interpo.Cskew1} and \ref{interpo.Cskew2}.

\begin{proof}[Proof of Theorem \ref{interpo.Cskew1}]
The implication (i)$\implies$(ii) is clear.

\medskip

(ii)$\implies$(iii). First, let us show that $V_1$ exists. Since $VN_kV^{-1}=-N_k^*$, one can use \eqref{conjugate.relation} to show that $VN_k^*V^{-1}=-N_k$ for every $k\in K$. Consequently, for any $n\in\Zp$, $p\in\P{2n}$ and $\{k_l\}_{l=1}^n\subseteq K$, one can easily see that
$$
Vp(N_{k_1},N_{k_1}^*,\dots,N_{k_n},N_{k_n}^*)=p(-N_{k_1},-N_{k_1}^*,\dots,-N_{k_n},-N_{k_n}^*)^*V.
$$
Therefore, for every $i\in I$,
\begin{align*}
Vp(N_{k_1},N_{k_1}^*,\dots,N_{k_n},N_{k_n}^*)x_i
&=p(-N_{k_1},-N_{k_1}^*,\dots,-N_{k_n},-N_{k_n}^*)^*Vx_i\\
&=p(-N_{k_1},-N_{k_1}^*,\dots,-N_{k_n},-N_{k_n}^*)^*y_i
\end{align*}
and
\begin{align*}
Vp(N_{k_1},N_{k_1}^*,\dots,N_{k_n},N_{k_n}^*)y_i
&=p(-N_{k_1},-N_{k_1}^*,\dots,-N_{k_n},-N_{k_n}^*)^*Vy_i\\
&=p(-N_{k_1},-N_{k_1}^*,\dots,-N_{k_n},-N_{k_n}^*)^*x_i.
\end{align*}
Put $\l_0=L_{\mathcal{N},X\cup Y}$. Then, the above equalities imply that $V\l_0=\l_0$. Since $V$ is anti-unitary, we get $V\l_1=V\l_0^{\perp}=\l_0^{\perp}=\l_1$.

Let $V_1$ denote the restriction of $V$ to $\l_1$. Then, $V_1$ is anti-unitary on $\l_1$, and the relations $VN_kV^{-1}=-N_k^*$, for every $k\in K$, imply $V_1N_{k\vert \l_1}V_1^{-1}=-\left(N_{k\vert \l_1}\right)^*$. This proves the first part.

Now, by Lemma \ref{TN=M*T}, for all Borel subsets $\Delta\subseteq\C$ and $k\in K$, we have $VE_{N_k}(-\Delta)=E_{N_k}(\Delta)V$. The remainder of the proof is now similar to \eqref{P.xi.xj} and \eqref{P.xi.yj}. 

\medskip

(iii)$\implies$(i). For each $k\in K$, put $M_k=-N_k$. Let $K_0,K_1\subseteq K$ be finite subsets, and let $\{\Delta_k\}_{k\in K_0}$ and $\{D_k\}_{k\in K_1}$ be Borel subsets of $\C$. Then, for all $i,j\in I$, we have
\begin{align*}
\left\langle \left(\prod_{k\in K_0} E_{N_k}(\Delta_k)\right) x_i, \left(\prod_{l\in K_1} E_{M_l}(D_l)\right)x_j \right\rangle
=\left\langle \left(\prod_{k\in K_0} E_{N_k}(\Delta_k)\right) x_i, \left(\prod_{l\in K_1} E_{-N_l}(D_l)\right)x_j \right\rangle\\
=\left\langle \left(\prod_{k\in K_0} E_{N_k}(\Delta_k)\right) x_i, \left(\prod_{l\in K_1} E_{N_l}(-D_l)\right)x_j\right\rangle\\
=\left\langle \left(\prod_{k\in K_0} E_{N_k}(-\Delta_k)\right) x_j, \left(\prod_{l\in K_1} E_{N_l}(D_l)\right)x_i \right\rangle\\
=\left\langle \left(\prod_{k\in K_0} E_{M_k}(\Delta_k)\right) x_j, \left(\prod_{l\in K_1} E_{N_l}(D_l)\right)x_i \right\rangle.
\end{align*}
Similarly,
$$
\left\langle \left(\prod_{k\in K_0} E_{N_k}(\Delta_k)\right) x_i, \left(\prod_{l\in K_1} E_{M_l}(D_l)\right)y_j \right\rangle
=\left\langle \left(\prod_{k\in K_0} E_{M_k}(\Delta_k)\right) x_j, \left(\prod_{l\in K_1} E_{N_l}(D_l)\right)y_i \right\rangle.
$$
It follows by Theorem \ref{interpo.Nk.Mk} that there exists a conjugation $C_0$ on $\l_0:=L_{\mathcal{N}, X\cup Y}$ such that $C_0{N_k}_{\vert \l_0}C_0=-\left({N_k}_{\vert \l_0}\right)^*$ and $C_0x_i=y_i$ for all $k\in K$ and $i\in I$.

Now, since by hypothesis $V_1{N_k}_{\vert \l_1}V_1^{-1}=-\left({N_k}_{\vert \l_1}\right)^*$ for every $k\in K$, Lemma \ref{cong.skewsym} implies that there exists a conjugation $C_1$ on $\l_1$ so that $C_1{N_k}_{\vert \l_1}C_1=-\left({N_k}_{\vert \l_1}\right)^*$ for every $k\in K$. Therefore, $C=C_0\oplus C_1$ is a conjugation on $\H=\l_0\oplus \l_1$ that satisfies all required equalities.
\end{proof}

\medskip

\begin{proof}[Proof of Theorem \ref{interpo.Cskew2}]
The implication (i)$\implies$(ii) follows directly from Theorem \ref{interpo.Cskew1}.

\medskip

(ii)$\implies$(i). Let $\H_0=L_{\{N\},X\cup Y}$ and $\H_1=\H_0^{\perp}$, and denote by $N_0$ and $N_1$ the restrictions of $N$ to $\H_0$ and $\H_1$, respectively. In the light of Theorem \ref{interpo.Cskew2} and Lemma \ref{cong.skewsym}, we only need to show that $N_{1}$ is skew-symmetric. This follows immediately from Theorem \ref{directsumskew} and the facts that $m\left({N_0}_{\vert \ker(N_0)^{\perp}}\right)<\infty$ and $N=N_0\oplus N_{1}$ is skew-symmetric.

\medskip

Let us now prove the second part of the theorem. Assume that $I$ is finite. Since
$$
\H_0=\bigvee\left(\sum_{i\in I} L_{\{N\},\{x_i\}}+L_{\{N\},\{y_i\}}\right)=\bigvee\left(\sum_{i\in I} L_{\{N_0\},\{x_i\}}+L_{\{N_0\},\{y_i\}}\right)
$$
and 
the restrictions of $N_0$ to $L_{\{N_0\},\{x_i\}}$ and $L_{\{N_0\},\{y_i\}}$ are star-cyclic, Lemma \ref{twostarcyclics} implies that $m\left({N_0}_{\vert \ker(N_0)^{\perp}}\right)<\infty$. This concludes the proof.
\end{proof}


\begin{thebibliography}{99}

\bibitem{Amara.Oudghiri} {\sc Z. Amara and M. Oudghiri}, {\it Finite rank perturbations of complex symmetric operators}, J. Math. Anal. Appl. {\bf 495} (2021), no. 1, 124720.


\bibitem{Amara.Oudghiri2} {\sc Z. Amara and M. Oudghiri}, {\it $C$-normality of rank-one perturbations of normal operators}, Linear Multilinear Algebra {\bf 71} (2023), no. 15, 2426--2440.



\bibitem{Bender} {\sc C. M. Bender}, {\it Making sense of non-Hermitian Hamiltonians}, Rep. Prog. Phys. {\bf 70} (2007), no. 6, 947--1018.


\bibitem{Bender.Boettcher} {\sc C. M. Bender and S. Boettcher}, {\it Real spectra in non-{H}ermitian {H}amiltonians having {$\mathscr{PT}$} symmetry}, Phys. Rev. Lett. {\bf 80} (1998), no. 24, 5243--5246.


\bibitem{Bu.Zhu} {\sc Q. Bu and S. Zhu}, {\it The Weyl-von Neumann theorem for skew-symmetric operators}, Ann. Funct. Anal. {\bf 14} (2023), no. 2, 43.


\bibitem{Camara.Klis.Lanucha.Ptak} {\sc M. C. C\^{a}mara, K. Kli\'{s}-Garlicka, B. \L anucha and M. Ptak}, {\it Conjugations in $L^2$ and their invariants}, Anal. Math. Phys. {\bf 10} (2020), no. 2, 22.


\bibitem{Camara.Klis.Lanucha.Ptak2} {\sc M. C. C\^{a}mara, K. Kli\'{s}-Garlicka, B. \L anucha and M. Ptak}, {\it Conjugations in {$L^2(\mathcal H)$}}, Integral Equations Operator Theory {\bf 92} (2020), no. 6, 48.


\bibitem {Cao.Hu}{\sc W. Cao and L. Hu}, {\it Projective interpolation of polynomial vectors and improved key recovery attack on {SFLASH}}, Des. Codes Cryptogr. {\bf 73} (2014), no. 3, 719--730.


\bibitem{Conway.Func.An} {\sc J. B. Conway}, {\it A Course in Functional Analysis}, Second edition, Grad. Texts Math. {\bf 96}, Springer-Verlag, New York (1990).


\bibitem{Conway.Op.Th} {\sc J. B. Conway}, {\it A Course in Operator Theory}, Grad. Stud. Math. {\bf 21}, American Mathematical Society, Providence, RI (2000).


\bibitem{Davidson} {\sc K. R. Davidson}, {{\it $C^*$}-algebras by example}, Fields Inst. Monogr. {\bf 6}, American Mathematical Society, Providence, RI (1996).


\bibitem{Dymek.Planeta.Ptak} {\sc P. Dymek, A. Planeta and M. Ptak}, {\it Conjugations preserving {T}oeplitz kernels}, Integral Equations Operator Theory {\bf 94} (2022), no. 4, 39.


\bibitem{Dymek.Planeta.Ptak2} {\sc P. Dymek, A. Planeta and M. Ptak}, {\it Conjugations on {$L^2$} space on the real line}, Results Math. {\bf 79} (2024), no. 1, 39.


\bibitem{Garcia.Putinar} {\sc S. R. Garcia and M. Putinar}, {\it Complex symmetric operators and applications}, Trans. Amer. Math. Soc. {\bf 358} (2006), no. 3, 1285--1315.


\bibitem{Garcia.Putinar2} {\sc S. R. Garcia and M. Putinar}, {\it Complex symmetric operators and applications II}, {Trans. Amer. Math. Soc.}, {\bf 359} (2007), no. 8, 3913--3931.


\bibitem{Garcia.Prodan.Putinar} {\sc S. R. Garcia, E. Prodan and M. Putinar}, {\it Mathematical and physical aspects of complex symmetric operators}, J. Phys. A {\bf 47} (2014), no. 35, 353001.


\bibitem{GUO.JI.ZHU} {\sc K. Guo, Y. Ji and S. Zhu}, {\it A {$C^*$}-algebra approach to complex symmetric operators}, Trans. Amer. Math. Soc. {\bf 367} (2015), no. 10, 6903--6942.


\bibitem{Li.Zhu} {\sc C. G. Li and S. Zhu}, {\it Skew symmetric normal operators}, Proc. Amer. Math. Soc. {\bf 141} (2013), no. 8, 2755--2762.


\bibitem{Liu.Shi.Wang.Zhu} {\sc T. Lui, L. Shi, C. Wang and S. Zhu}, {\it An interpolation problem for conjugations}, J. Math. Anal. Appl. {\bf 500} (2021), no. 1, 125118.


\bibitem{Liu.Xie.Zhu} {\sc T. Liu, X. Xie and S. Zhu}, {\it An interpolation problem for conjugations {II}}, Mediterr. J. Math. {\bf 19} (2022), no. 4, 153.


\bibitem{Mashreghi.Ptak.RossI} {\sc J. Mashreghi, M. Ptak and W. T. Ross}, {\it Conjugations of unitary operators, {I}}, Anal. Math. Phys. {\bf 14} (2024), no. 3, 62.


\bibitem{Mashreghi.Ptak.RossII} {\sc J. Mashreghi, M. Ptak and W. T. Ross}, {\it Conjugations of unitary operators, {II}}, Anal. Math. Phys. {\bf 14} (2024), no. 3, 56.


\bibitem{Mehl}{\sc C. Mehl}, {\it Condensed forms for skew-{H}amiltonian/{H}amiltonian pencils}, SIAM J. Matrix Anal. Appl. {\bf 21} (1999), no. 2, 454--476.


\bibitem{Pinero.Singh}{\sc F. Pi\~{n}ero and P. Singh}, {\it The weight spectrum of certain affine {G}rassmann codes}, Des. Codes Cryptogr. {\bf 87} (2019), no. 4, 817--830.


\bibitem{Prodan.Garcia.Putinar} {\sc E. Prodan, S. R. Garcia and M. Putinar}, {\it Norm estimates of complex symmetric operators applied to quantum systems}, J. Phys. A {\bf 39} (2006), no. 2, 389--400.




\bibitem{Radjavi.Rosenthal} {\sc H. Radjavi and P. Rosenthal}, {\it Invariant subspaces}, Ergeb. Math. Grenzgeb., Band 77[Results in Mathematics and Related Areas] (1973).


\bibitem{Wang.Xie.Yan.Zhu} {\sc C. Wang, X. Xie, S. Yan and S. Zhu}, {\it An interpolation problem for conjugations {III}}, Banach J. Math. Anal. {\bf 16} (2022), no. 4, 61.


\bibitem{Wang.Zhao.Zhu} {\sc C. Wang, J. Zhao and S. Zhu}, {\it Remarks on the structure of {$C$}-normal operators}, Linear Multilinear Algebra {\bf 70} (2022), no 9, 1682--1696.


\bibitem{Zhu14} {\sc S. Zhu}, {\it Approximate unitary equivalence to skew symmetric operators}, Complex Anal. Oper. Theory {\bf 8} (2014)0, no. 7, 1565--1580.


\bibitem{Zhu.Li} {\sc S. Zhu and C. G. Li}, {\it Complex symmetry of a dense class of operators}, Integral Equations Operator Theory {\bf 73} (2012), no. 2, 255--272.

\end{thebibliography}
\end{document}